\documentclass[12pt, reqno]{amsart}
\allowdisplaybreaks
\usepackage{amssymb,amsthm,amsmath,bm}
\usepackage[numbers,sort&compress]{natbib}

\title[]{Initial-boundary value problem to 2D Boussinesq equations for MHD convection with stratification effects}
\author{Dongfen Bian}
\address[Dongfen Bian]{School of mathematics and statistics, Beijing Institute of Technology, Beijing, 100081, China.}
\address[]{Beijing Key Laboratory on MCAACI, Beijing Institute of Technology, Beijing, 100081, China.}
\email{biandongfen@bit.edu.cn}
\author{Jitao Liu}
\address[Jitao Liu]{College of Applied Sciences, Beijing University of Technology, Beijing, 100124, China.}
\email{jtliu@bjut.edu.cn, jtliumath@qq.com}


\theoremstyle{plain}
\newtheorem{corollary}{Corollary}[section]
\newtheorem{theorem}{Theorem}[section]
\newtheorem{lemma}{Lemma}[section]
\newtheorem{proposition}{Proposition}[section]

\theoremstyle{definition}
\newtheorem{definition}{Definition}[section]

\newtheorem{remark}{Remark}[section]

\let\f=\frac
\let\p=\partial
\def\R{\Bbb R}

\def\grad{\nabla}
\def\dv{\mbox{div}}
\def\no{\noindent}
\def\endproof{\hphantom{MM}\hfill\llap{$\square$}\goodbreak}
\newcommand{\beq}{\begin{equation}}
\newcommand{\eeq}{\end{equation}}
\newcommand{\ben}{\begin{eqnarray}}
\newcommand{\een}{\end{eqnarray}}
\newcommand{\beno}{\begin{eqnarray*}}
\newcommand{\eeno}{\end{eqnarray*}}

\begin{document}


\begin{abstract}
This paper is concerned with the initial-boundary value problem to 2D magnetohydrodynamics-Boussinesq system with the temperature-dependent viscosity, thermal diffusivity and electrical conductivity. First, we establish the global weak solutions under the minimal initial assumption. Then by imposing higher regularity assumption on the initial data, we  obtain
the global strong solution with uniqueness. Moreover, the exponential decay estimate of the solution is obtained.
\end{abstract}
\maketitle

\section{Introduction and main results}\hspace*{\parindent}

In this paper, we consider the following 2D incompressible Boussinesq equations for magnetohydrodynamics (MHD) convection with stratification effects \cite{P-B-M-2013, Bian-Gui, Bian-Guo-Gui-Xin}:
\begin{equation}\label{B-MHD}
	\begin{cases}
		\partial_t \theta + u\cdot\nabla\theta-	\mbox{div} (\kappa(\theta)\nabla\theta))=-u_2 T_0'(x_2), \\
		\partial_t u +  u\cdot\nabla u -\dv (\mu(\theta)
		\nabla u)+
		\grad \Pi= \theta e_2+J \,B^{\perp}, \\
		\partial_t B+u \cdot \grad B-\grad^{\perp}(\sigma(\theta)J)=B\cdot \grad u, \\
		\mbox{div}  u =\dv\, B= 0.
	\end{cases}
\end{equation}
The unknowns are the temperature $\theta$ (or the density in the modeling of geophysical fluids), the solenoidal velocity field $u=(u_1, u_2)$, the magnetic field $B=(B_1,B_2)$, and the scalar pressure $\Pi$. Here the current density $J = \grad^{\perp} \cdot B$, $\grad^{\perp}=(-\partial_2, \partial_1)^{T}$ and $e_2=(0,1)$. In addition, we denote here by  $\sigma(\theta)$  the electrical conductivity of the fluid, $\mu(\theta)$ the fluid viscosity, and $\kappa(\theta)$ the thermal diffusivity. All of them are assumed to be smooth in $\theta$ and satisfy
\begin{equation}\label{Assumption}
\kappa(\theta ), \mu(\theta ), \sigma(\theta )\geq C_0^{-1}.
\end{equation}

Physically, the first equation of $\eqref{B-MHD}$ describes the temperature fluctuation in which the term  $-u_2T_0'(x_2)$ stratification effects about a linear mean temperature profile $T_0(x_2)$ in the direction of gravity \cite{P-B-M-2013}. The second equation of  $\eqref{B-MHD}$ represents the conservation law of the momentum with the effect of the buoyancy $\theta e_2 $ and the Lorentz force $J\, B^{\perp}$. The last equation of  $\eqref{B-MHD}$ shows that the electromagnetic field is governed by the Maxwell equation. The sign of $T_0'(x_2)$  that appears in the equation of the temperature $\theta$ is critical (cf.\cite{Majda-2002}). For the case $T_0'(x_2)<0$, the situation is unstable because the hot fluid at the bottom is less dense than the fluid above it. While for the case $T_0'(x_2)>0$, the density decreases with height and the heavier fluid is below lighter fluid. This is the situation of stable stratification, and the real quantity $\mathcal{N}(x_2) :=\sqrt{T_0'(x_2)}$ is called the buoyancy or Brunt-V\"{a}is\"{a}r\"{a} frequency (stratification-parameter) \cite{Ib-Yo, Majda-2002}. In one word, the system \eqref{B-MHD} is a combination of the incompressible Boussinesq equations of fluid dynamics and Maxwell's equations of electromagnetism, where the displacement current is neglected \cite{ku-ly, lau-li}.

When the fluid is not affected by the temperature and stratification, that is, $\theta \equiv 0$ and $T_0(x_2) \equiv {\rm Const.}$, then the equations  \eqref{B-MHD} become the standard MHD system and govern the dynamics of the velocity and the magnetic field in electrically conducting fluids such as plasmas and reflect the basic physics conservation laws. There have been a lot of studies on MHD by physicists and mathematicians. For instance, G. Duvaut and J. L. Lions \cite{du-lions} established the local existence and uniqueness of solutions in the Sobolev spaces $H^s(\mathbb{R}^d)$, $s \geq d$. Besides, the global existence of solutions for small initial data is also proved in this paper. Then M. Sermange and R. Temam  \cite{ser-te} examined the properties of these solutions. In particular, for two dimensional case, the local strong solution has been proved to be global and unique. Recent work on the MHD equations developed regularity criteria in terms of the velocity field and dealt with the MHD equations with dissipation and magnetic diffusion (see, e.g.
\cite{chen-m, he-xin, he-xin2}). Also the issue of global regularity on the MHD equations with partial dissipation, has been extensively studied (see, e.g., \cite{cao-wu, cao-wu-yuan, L-X-Z-2013, lin-zh, R-W-X-Z-2014, hu-lin, lei, Zhang-ting2}). Further background and motivation for the MHD system may be found in \cite{chen-m, des-br, du-lions, ger-br, Gui-2014, he-xin, ser-te, lei2005, lei-zhou, wu-wu-xu, Xu-Zhang, Zhang-ting} and references therein.

On the other hand, if the fluid is not affected by the Lorentz force and stratification, that is, $B \equiv 0$ and $T_0(x_2) \equiv {\rm Const.}$, then the equations  \eqref{B-MHD} become the classical Boussinesq system. In \cite{Dan-Pai-2008}, R. Danchin and M. Paicu obtained the global existence
of weak solution for $L^2$ initial data. Started from D. Chae, T. Y. Hou and C. Li \cite{Chae-2006, H-L-2005},  there are many works devoted to the 2D Boussinesq system with partial constant viscosity, one can also refer to \cite{A-H-2007, CaoWu3,ChenLiu,Danchin,H-K-2009,JiuLiu,L-P-Z-2011,WZ} for related works. Regarding the Boussinesq system with temperature-dependent viscosity and thermal diffusivity, Wang-Zhang \cite{zhifeizhang}
proved the global well-posed of Cauchy problem for smooth initial data , see also \cite{SunZhang} for initial-boundary value problem. This result was then generalized to the case without viscosity by Li-Xu \cite{Li-Xu-2013} and Li-Pan-Zhang \cite{LiPanZhang} for the whole space and bounded domain separately.

For Boussinese-MHD system \eqref{B-MHD} with the temperature-dependent  viscosity, thermal diffusivity  and  electrical conductivity,  Bian-Gui \cite{Bian-Gui} and Bian-Guo-Gui-Xin \cite{Bian-Guo-Gui-Xin} rigorously justified the stability and instability in a fully nonlinear, dynamical setting from  mathematical point of view in unbounded domain. However, in real world, the flows often move in bounded domains with constraints from boundaries,
where the initial boundary value problems appear. Compared with Cauchy problems, solutions of the initial boundary
value problems usually exhibit different behaviors and much richer phenomena. In \cite{Bian}, the author obtained the global well-posedness for Boussinese-MHD system \eqref{B-MHD} in bounded domain with constant viscosity. Nevertheless, for initial-boundary value problem to the system \eqref{B-MHD} with temperature-dependent viscosity, it is still open.

In this paper, we will investigate the initial-boundary value problem to the system \eqref{B-MHD} with the temperature-dependent viscosity, thermal diffusivity and electrical conductivity in a bounded domain. Without loss of generality, we take $\mathcal{N}=1$ which does not change the results of original model. Under this assumption, the system  \eqref{B-MHD} is reformulated as \begin{equation}\label{B-MHD-equ}
 	\begin{cases}
 		\p_t \theta + u\cdot\nabla\theta-\dv(\kappa(\theta)\nabla\theta))=-u_2, \\
 		\p_t u + u\cdot\nabla u -\dv (\mu(\theta)
 		\nabla u)+
 		\grad \Pi=\theta e_2+B\cdot\nabla B, \\
 		\p_t B+u \cdot \grad B-\grad^{\perp}(\sigma(\theta)J)=B\cdot \grad u, \\
 		\dv\, u =\dv\, B= 0.
 	\end{cases}
 \end{equation}
What's more, we will treat \eqref{B-MHD-equ} with prescribed initial conditions:
 \begin{align}\label{initial}
 	\left(\theta,u,B\right)(x,0)=(\theta_0,u_0,B_0),~~ x\in \Omega,
 \end{align}
 and physical boundary conditions:
 \begin{align}
 	&\theta|_{\p\Omega}=0,\quad u|_{\partial \Omega}=0,\quad B|_{\p\Omega}=0.\label{B-MHD-boundary-B}
 \end{align}
In addition, we also require the following compatibility conditions
\begin{equation}\label{eq3}
\left\{\begin{array}{ll}
\theta_0|_{\p \Omega}=u_0|_{\p\Omega}=B_0|_{\p\Omega}=0,&\\
\nabla\cdot u_0=\nabla\cdot B_0=0,&\\
{[u_0\cdot\nabla \theta_0]}|_{\p \Omega}={[\nabla\cdot(\kappa(\theta_0)\nabla\theta_0)-u_{2}(0,x)]}|_{\p \Omega},\\
{[u_0\cdot\nabla u_0+\nabla\pi_0]}|_{\p \Omega}={[\nabla\cdot(\mu(\theta_0)\nabla u_0)+B_0\cdot\nabla B_0+\theta_0e_2]}|_{\p \Omega},\\
{[u_0\cdot\nabla B_0]}|_{\p \Omega}={[\nabla^{\perp}(\sigma(\theta_0)\nabla^{\perp}\cdot  B_0)+B_0\cdot\nabla u_0]}|_{\p \Omega}.
\end{array}\right.
\end{equation}
Here $\pi_0$ is determined by the divergence-free condition $\nabla\cdot u_0=0$ with the Neumann boundary condition
\begin{equation}\label{eq4}
\nabla \pi_0\cdot{\bm n}|_{\p \Omega}=[\nabla\cdot(\mu(\theta_0)\nabla u_0)+B_0\cdot\nabla B_0+\theta_0e_2-u_0\cdot\nabla u_0]\cdot{\bm n}|_{\p \Omega}
\end{equation}
and ${\bm{n}}$ denotes the unit outward normal on $\partial\Omega$.

Now, we are in the position to state the main results of this paper. Our first result is concerning the solvability for the weak solution of \eqref{B-MHD} with the initial data in the energy spaces.

\begin{theorem}\label{T1}
Let $\Omega\subset\R^2$ be a bounded domain with smooth boundary. Assume $(\theta_0,\,u_0,\,B_0)\in H_0^{1}(\Omega)$, then \eqref{B-MHD-equ}-\eqref{B-MHD-boundary-B} has a global weak solution in the sense of Definition \ref{weak}. Moreover, the solution has the following decay estimate
\begin{equation}\label{u-B-estimate}
\left\|\theta(\cdot,t),\,u(\cdot,t),\,B(\cdot,t)\right\|_{H^1(\Omega)}^2\leq Ce^{-\alpha t},~~ \forall~ t\geq 0,
\end{equation}
where $C$ and $\alpha$ are the constants depending on $C_0$, $\Omega$ and $\left\|\theta_0,\,u_0,\,B_0\right\|_{H^1(\Omega)}$.
\end{theorem}

If we further impose higher regularity assumption on the initial data, one can then get the strong solution with uniqueness. It should be pointed out that in this theorem, there is not any smallness restriction upon the initial data.

\begin{theorem}\label{T2}
Suppose $\Omega\subset\R^2$ be a bounded domain with smooth boundary, $(\theta_0,\,u_0,\,B_0)$\\
$\in H^2(\Omega)$ are vector fields such that (\ref{eq3}) and (\ref{eq4}) hold. Then the system \eqref{B-MHD-equ}-\eqref{B-MHD-boundary-B} has a unique global strong solution $(\theta,\,u,\,b)$ which satisfies
\beno
(\theta_t,\,u_t,\,B_t)\in C(0,T;L^{2}(\Omega))\cap L^2(0,T;H_0^{1}(\Omega))
\eeno
and
\beno
(\theta,\,u,\,B)\in C(0,T;H^{2}(\Omega))\cap L^2(0,T;H^{3}(\Omega))
\eeno
for any $T>0$. In addition, the corresponding solution has the exponential decay rate
\begin{equation}\label{u-B-estimate1}
\left\|\theta(\cdot,t),\,u(\cdot,t),\,B(\cdot,t)\right\|_{H^2(\Omega)}^2\leq Ce^{-\alpha t},~~~\forall~ t\geq 0,
\end{equation}
where $C$ and $\alpha$ are the constants depending on $C_0$, $\Omega$ and $\left\|\theta_0,\,u_0,\,B_0\right\|_{H^2(\Omega)}$.
\end{theorem}

\begin{remark}\label{rmk-1.1}
Theorem 1.2 also implies that for the Boussinesq system addressed in \cite{SunZhang}, the decay rate of resulting solution is $\|\theta(\cdot,t),\,u(\cdot,t)\|_{H^2(\Omega)}^2\leq Ce^{-\alpha t}$ for any $t\geq 0$.
\end{remark}

\begin{remark}\label{rmk-1.2}
When the Dirichlet boundary condition $B|_{\p\Omega}=0$ be replaced by $ B\cdot {\bm n}|_{\partial \Omega}=\nabla^{\perp}\cdot B|_{\partial \Omega}=0$ (the perfectly conducting wall condition), the results in Theorems \ref{T1} and \ref{T2} still hold.
\end{remark}

\begin{remark}\label{rmk-1.3}
If one takes the adiabetic boundary condition $\partial_n\theta|_{\partial \Omega}=0$ instead of $\theta|_{\p\Omega}=0$, the contents in Theorems \ref{T1} and \ref{T2} hold except for the $L^2$ decay rate of temperature $\|\theta(\cdot,t)\|_{L^2(\Omega)}^2\leq Ce^{-\alpha t}$.
\end{remark}

The proof of main theorems is divided into two main steps. The first step is to establish the global
existence of weak solutions which are defined as follows.
\begin{definition}\label{weak}
Suppose $(\theta_0,u_0,B_0,)\in L^2(\Omega)$, a pair of measurable vector field $\theta(x,t),\,u(x,t)$ and $B(x,t)$  is called a weak solution of \eqref{B-MHD-equ}-\eqref{B-MHD-boundary-B} if
\beno
(1)&&(\theta(x,t),\,u(x,t),\,B(x,t))\in C(0,T;L^2(\Omega))\cap L^2(0,T;H^1(\Omega));\\
(2)&&\int_{\Omega}\theta_0\phi_0dx+\int_0^T\int_{\Omega}\big(\theta\phi_t+u\cdot\nabla\phi\theta-\kappa(\theta)\nabla\theta\cdot\nabla\phi-u_2\phi\big)dxdt=0,\\
&&\int_{\Omega}u_0\cdot\psi_0dx+\int_0^T\int_{\Omega}\big()u\cdot\psi_t+u\cdot\nabla\psi\cdot u-B\cdot\nabla\psi\cdot B-\mu(\theta)\nabla u\cdot\nabla \psi\big)dxdt\\
&&=\int_0^T\int_{\Omega}\theta e_2\cdot\psi\,dxdt,\\
&&\int_{\Omega}B_0\cdot\varphi_0dx+\int_0^T\int_{\Omega}\big(B\cdot\varphi_t+u\cdot\nabla\varphi\cdot B-\sigma(\theta)J\grad^{\perp}\cdot \varphi\big)dxdt\\
&&=\int_0^T\int_{\Omega}B\cdot\nabla\varphi\cdot u\,dxdt
\eeno
holds for any $(\phi,\,\psi,\,\varphi)\in C^\infty(\Omega\times[0,T])$ with $\phi|_{\p\Omega}=\phi(x,T)=0$, $\nabla\cdot\psi=\psi|_{\p\Omega}=\psi(x,T)=0$ and $\nabla\cdot\varphi=\varphi|_{\p\Omega}=\varphi(x,T)=0$.
\end{definition}
\begin{remark}\label{gw1}
Following standard arguments as in the theory of the Navier-Stokes equations (see e.g., \cite{RT}), it is clear that the above system is equivalent to the system
\ben\label{w4}
\f{d}{dt}<\theta,\phi>+(\kappa(\theta)\nabla\theta,\nabla\phi)+b(u,\theta,\phi)=-(u_2,\phi),
\een
\ben\label{w5}
\f{d}{dt}<u,\psi>+(\mu(\theta)\nabla u,\nabla\psi)+b(u,u,\psi)=b(B,B,\psi)+(\theta e_2,\psi),
\een
\ben\label{w6}
\f{d}{dt}<B,\varphi>+(\sigma(\theta)\grad^{\perp}\cdot B,\grad^{\perp}\cdot\varphi)+b(u,B,\varphi)=b(B,u,\varphi)
\een
for any $\phi\in L^2(0,T; H_0^1(\Omega))$ and $\psi,\,\varphi\in L^2(0,T; V)$.
\end{remark}

The second step is to build up the higher estimates and the uniqueness of solution by a priori estimates under the initial and boundary conditions \eqref{initial}-\eqref{B-MHD-boundary-B}. More precisely, we will do the $L^\infty(0,T;H^{2}(\Omega))$ estimates of temperature, velocity field and magnetic field for any $T>0$. Due to the strong coupling in the nonlinearities and the boundary effects, there are not enough spatial derivatives of the solution at the boundary. To solve it,
we will make full use of the Sobolev embeddings and classical regularity results of elliptic equations to obtain the estimates of high-order spatial derivatives, which is distinguished from the Cauchy problem in \cite{Bian-Gui}. Our energy estimates is somewhat delicate. In the end, we got the the desired estimates which lead to the global regularity and uniqueness of solution.

This paper is organized as follows. In section 2, we introduce some useful Propositions and Lemmas of this paper. In section 3, we will concentrate on the global weak solution (i.e., the proof of Theorem \ref{T1}). Section 4 is devoted to the global strong solution (i.e., the proof of Theorem \ref{T2}).

\section{Preliminary}\hspace*{\parindent}
\subsection{Notations}\hspace*{\parindent}
In this section, we will give some Propositions and Lemmas which will be used to prove Theorem \ref{T1}. Initially, we define the inner products on $L^2(\Omega)$ and space $V$ by
$$(u,v)=\sum\limits_{i=1}^2\int_{\R^2}u_iv_i\,dx,$$
and
$$V=\{u\in H_0^1(\Omega): \nabla\cdot u = 0\,\,\hbox{in}\,\,\Omega\},$$
respectively. Then, we will denote by $V'$ the dual space of $V$ and the action of $V'$ on $V$ by $<\cdot\,,\,\cdot>$. Moreover, we use the following notation for the trilinear continuous form by setting
\ben\label{b0}
b(u,v,w)=\sum\limits_{i,j=1}^2\int_{\Omega}u_i\p_iv_jw_j\,dx.
\een
If $u\in\,V$, then
\ben\label{b1}
b(u,v,w)=-b(u,w,v),\,\,\,\forall\,\,v,w\in H_0^1(\Omega),
\een
and
\ben\label{b2}
b(u,v,v)=0,\,\,\,\forall\,\,v\in H_0^1(\Omega).
\een

The following one to be introduced is the well known Gagliardo-Nirenberg interpolation inequality.
\begin{proposition}\label{P1}{\cite{NIR}}
Let $f(x)$ be a function defined on a bounded domain $\Omega\subset\R^n$ with smooth boundary, fix $1\leq q, r \leq\infty $ and a natural number m. Suppose also that a real number $\alpha$ and a natural number $j$ are such that
$$\frac{1}{p} = \frac{j}{2} + \left( \frac{1}{r} - \frac{m}{n} \right) \alpha + \frac{1 - \alpha}{q},$$
and
$$\frac{j}{m} \leq \alpha \leq 1,$$
then there holds that
$$\| \mathrm{D}^{j} f \|_{L^{p}(\Omega)} \leq C_{1} \| \mathrm{D}^{m} f \|_{L^{r}(\Omega)}^{\alpha} \| f \|_{L^{q}(\Omega)}^{1 - \alpha} + C_{2} \| f \|_{L^{s}(\Omega)},$$
where $s > 0$ is arbitrary and the constants $C_1$ and $C_2$ depend upon $\Omega,\,m,\,j,\,s$ only.
\end{proposition}
By inputting $n=2$ and $p=4,\infty$ separately, it is clear to derive the following corollary.
\begin{corollary}\label{C1}
Suppose $\Omega\subset\R^2$ be a bounded domain with smooth boundary. Then

(1)\,\,\,$\| f \|_{L^{4}(\Omega)} \leq C\, (\|  f \|_{L^{2}(\Omega)}^{\f12} \| \nabla f \|_{L^{2}(\Omega)}^{\f12} + \| f \|_{L^{2}(\Omega)}),\,\,\,\forall f\in H^1(\Omega);$\vskip 0.2cm

(2)\,\,\,$\| \nabla f \|_{L^{4}(\Omega)} \leq C\, (\|  f \|_{L^{2}(\Omega)}^{\f14} \| \nabla^2 f \|_{L^{2}(\Omega)}^{\f34} + \| f \|_{L^{2}(\Omega)}),\,\,\,\forall f\in H^2(\Omega);$\vskip 0.2cm

(3)\,\,\,$\| f \|_{L^{\infty}(\Omega)} \leq C\, (\|  f \|_{L^{2}(\Omega)}^{\f12} \| \nabla^2 f \|_{L^{2}(\Omega)}^{\f12} + \| f \|_{L^{2}(\Omega)}),\,\,\,\forall f\in H^2(\Omega);$\vskip 0.2cm

(4)\,\,\,$\| f \|_{L^{\infty}(\Omega)} \leq C\, (\|  f \|_{L^{2}(\Omega)}^{\f23} \| \nabla^3 f \|_{L^{2}(\Omega)}^{\f13} + \| f \|_{L^{2}(\Omega)}),\,\,\,\forall f\in H^3(\Omega).$
\end{corollary}

Then, let us recall some classical results which can be found in the cited reference.
\begin{lemma}\label{elliptic}\cite{evans,GT}
Let $\Omega\subset\R^2$ be a bounded domain with smooth boundary and consider the elliptic boundary-value problem
\begin{equation}\label{eel}
\left\{\begin{array}{ll}
-\Delta f=g\quad&\hbox{in}\,\,\Omega,\\
f=0\quad&\hbox{on}\,\,\p{\Omega}.
\end{array}\right.
\end{equation}
Then for any $p\in(1,\infty)$, integers $m\geq -1$ and $g\in W^{m,p}(\Omega)$,  (\ref{eel}) has a unique solution $f$ satisfying
$$
\|f\|_{W^{m+2,p}(\Omega)}\leq C\|g\|_{W^{m,p}(\Omega)},
$$
where $C$ depending only on $\Omega, \,m$ and $p$.
\end{lemma}

Now we set the coefficient $\mu(x)$ and $\kappa(x)$ be smooth functions satisfying
$$0<C_{\rm min}\leq\mu(x),\kappa(x)\leq C_{\rm max}<\infty.$$
Under this assumption, we then introduce the following four Lemmas.

\begin{lemma}\label{stokes1}\cite{Solonnikov,SunZhang}
Consider the Stokes system with variable coefficient in a bounded smooth domain $\Omega\subset\R^2:$
\beno
\left\{\begin{array}{ll}
-{\rm div}(\mu(x)\nabla u)+\nabla\pi=f\quad&\hbox{in}\,\,\Omega,\\
{\rm div}\,u=0\quad&\hbox{in}\,\,\Omega,\\
u=0\quad&\hbox{on}\,\,\p{\Omega}.
\end{array}\right.
\eeno
Then for any $f\in H^{-1}(\Omega)$, there exists a unique weak solution $(u,\pi)\in H_0^1(\Omega)\times L^2(\Omega)$ with $\int_\Omega\pi(x)dx=0$ satisying $$\|u\|_{H^{1}(\Omega)}+\|\pi\|_{L^{2}(\Omega)}\leq C\|f\|_{H^{-1}(\Omega)},$$
where the constant $C$ depends only on $C_{\rm min},\,C_{\rm max}$ and $\Omega$.
\end{lemma}

\begin{lemma}\label{stokes2}\cite{Solonnikov,SunZhang}
Let $(u,\pi)\in H^2(\Omega)\times H^1(\Omega)$ be a solution of the Stokes system of non-divergence form
\beno
\left\{\begin{array}{ll}
-\mu(x)\Delta u+\nabla\pi=f\quad&\hbox{in}\,\,\Omega,\\
{\rm div}\,u=0\quad&\hbox{in}\,\,\Omega,\\
u=0\quad&\hbox{on}\,\,\p{\Omega}.
\end{array}\right.
\eeno
Then there exists a constant $C=C(C_{\rm min},\,C_{\rm max},\,\Omega)$ such that
$$\|\nabla^2 u\|_{L^{2}(\Omega)}+\|\nabla \pi\|_{L^{2}(\Omega)}\leq C(\|f\|_{L^{2}(\Omega)}+\|u\|_{H^{1}(\Omega)}+\|\pi\|_{L^{2}(\Omega)}).$$
\end{lemma}

\begin{corollary}\label{stokes3}\cite{Solonnikov,SunZhang}
For the solution $(u,\pi)$ in Lemma \ref{stokes2}, if one further assume $f\in H^1(\Omega)$, then it holds that
$$\|\nabla^3 u\|_{L^{2}(\Omega)}+\|\nabla^2 \pi\|_{L^{2}(\Omega)}\leq C(\|f\|_{H^{1}(\Omega)}+\|u\|_{H^{2}(\Omega)}+\|\pi\|_{H^{1}(\Omega)}),$$
here $C$ depends only on $C_{\rm min},\,C_{\rm max}$ and $\Omega$.
\end{corollary}

\begin{lemma}\label{Boussinesq}\cite{SunZhang}
Suppose $\Omega\subset\R^2$ be a bounded domain with smooth boundary and
consider the initial-boundary value problem
\beno
\left\{\begin{array}{ll}
u_t+u\cdot\nabla u-\nabla\cdot(\mu(\theta)\nabla u)+\nabla p=\theta e_2+f,\\
\theta_t+u\cdot\nabla\theta-\nabla\cdot(\kappa(\theta)\nabla \theta)=g,\\
\nabla\cdot u=0,\\
(u,\theta)|_{t=0}=(u_0,\theta_0)(x),~~(u,\theta)|_{\partial \Omega}=0.
\end{array}\right.
\eeno
Assume $u_{0},\theta_0\in C^{1+\gamma}(\Omega)$ with $0<\gamma<1$ satisfying $\nabla\cdot u_{0}(x)=0$ and $(u_{0},\theta_0)|_{\partial \Omega}=0.$ Then for
any $f,g\in C([0,T];C^{1+\gamma}(\Omega))$, there exists a unique solution $(u,\theta,p)$ such that $(u,\theta,p)\in C^{1+\gamma}(\Omega\times[0,T])$.
\end{lemma}

\begin{lemma}\label{inner}
Let $\Omega\subset\R^2$ be a bounded domain with smooth boundary, $f:\,\Omega\rightarrow\R^2$ and $\phi:\,\Omega\rightarrow\R$ be the vector field and function respectively, then if follows that
$$\int_{\Omega}{\nabla^\perp}\cdot f\,\phi\,dx=\int_{\Omega}f\cdot{\nabla^\perp}\phi\,dx+\int_{\p\Omega}f\cdot{{\bm n}^\perp}\phi\,ds,$$
where $\perp$ is defined as $f^\perp=(-f_2,f_1)$.
\end{lemma}

\begin{lemma}\label{curlB}
Let $\Omega\subset\R^2$ be a bounded domain with smooth boundary, $f:\,\Omega\rightarrow\R^2$ be the vector field, then if holds that
$$\|\nabla f\|_{L^2(\Omega)}=\|\nabla^\perp\cdot f\|_{L^2(\Omega)}.$$
\end{lemma}
\no{\bf Proof.}\quad
By noticing $\nabla^\perp\nabla^\perp\cdot f=\Delta f$, by taking $\phi=\nabla^\perp\cdot f$ in Lemma \ref{inner}, one can prove this Lemma easily.
\endproof
To simplify the proofs of Theorems, it is better to introduce a new quantity
\ben\label{hatthetadef}
\hat{\theta}=\int_0^{\theta}\kappa(z)dz,
\een
which satisfies, after multiplying $\kappa(\theta)$ on both sides of $\eqref{B-MHD}^1$, that
\ben\label{simtheta}
\p_t\hat{\theta}+u\cdot\nabla\hat{\theta}-\kappa(\theta)\Delta\hat{\theta}=-\kappa(\theta)u_2,
\een
with the following initial and boundary conditions
\begin{equation}\label{ibsimtheta}
\left\{\begin{array}{ll}
\hat{\theta}(x,0)=\int_0^{\theta_0(x)}\kappa(z)dz\triangleq\hat{\theta}_0(x)\quad\quad&\hbox{in}\,\,\Omega,\\
\hat{\theta}=0\quad\quad&\hbox{on}\,\,\p\Omega.
\end{array}\right.
\end{equation}

\section{Global weak solution}\hspace*{\parindent}
In this section, we will make the effort to get the global weak solution. To start with, we build up the desired estimates mentioned in the introduction.
\begin{proposition}\label{estimates}
Let $(\theta_0,\,u_0,\,B_0)\in H_0^{1}(\Omega)$ and $\Omega$ be a bounded domain with smooth boundary. Suppose $(\theta,\,u,\,B)$ solves the system \eqref{B-MHD-equ}-\eqref{B-MHD-boundary-B}, then there holds that
\beno
\|\theta(\cdot,t),\,u(\cdot,t),\,B(\cdot,t)\|_{H^1(\Omega)}^2\leq Ce^{-\alpha t}
\eeno
and
\beno
\int_0^te^{\alpha \tau}\|{\theta}_{\tau},\,u_{\tau},\,B_{\tau}\|_{H^2(\Omega)}^2d\tau\leq C,
\eeno
for any $t>0$, where $C$ and $\alpha$ depend only on $C_0$, $\Omega$ and $\left\|\theta_0,\,u_0,\,B_0\right\|_{H^1(\Omega)}$.
\end{proposition}
The proof of Proposition \ref{estimates} is based on all the following subsections. Moreover, for any $t>0$, we will restrict the time to be within the interval $[0,t]$ in the rest of this section unless otherwise specified.
\subsection{$L^2$ Estimates}\hspace*{\parindent}
\begin{lemma}\label{L31}
Under the assumptions of Proposition \ref{estimates}, $\forall\,\,t\geq0$, there holds that
\beno
\|\theta(\cdot,t),\,u(\cdot,t),\,B(\cdot,t)\|_{L^2(\Omega)}^2\leq e^{-2\alpha t}\|\theta_0,\,u_0,\,B_0\|_{L^2(\Omega)}^2
\eeno
and
\beno
\int_0^te^{\alpha \tau}\|\nabla\theta,\,\nabla u,\,\nabla B\|_{L^2(\Omega)}^2d\tau\leq\frac{C_0}{2}\|\theta_0,\,u_0,\,B_0\|_{L^2(\Omega)}^2,
\eeno
where $\alpha=(C_0C^*)^{-1}$ with $C^*$ be the constant in Poincar\'{e} inequality for the domain $\Omega$.
\end{lemma}
\no{\bf Proof.}\quad
Multiplying $\eqref{B-MHD-equ}^1$ with $\theta$ and taking the inner product of $\eqref{B-MHD-equ}^2$ and $\eqref{B-MHD-equ}^3$ with $u$ and $B$ respectively, noticing \eqref{B-MHD-boundary-B}, Lemma \ref{curlB} and the fact that
\beno
&&\int_{\Omega}B\cdot \nabla B\cdot udx+\int_{\Omega}B\cdot \nabla u\cdot Bdx\\
&=&\int_{\Omega}\sum\limits_{i,j=1}^{2} (B_i\p_iB_ju_j+B_i\p_iu_jB_j)dx\\
&=&\int_{\Omega}\sum\limits_{i,j=1}^{2} \p_i(B_iB_ju_j+B_iu_jB_j)dx\\
&=&\int_{\p\Omega}(B\cdot {\bm n})(u\cdot B)=0,
\eeno
one has
\ben\label{L2}
\f{d}{dt}\|\theta,\,u,\,B\|_{L^2(\Omega)}^2+2C_0^{-1}\|\nabla\theta,\,\nabla u,\,\nabla B\|_{L^2(\Omega)}^2\leq 0.
\een
Considering the boundary condition $\theta|_{\p\Omega}=u|_{\p\Omega}=B|_{\p\Omega}=0$, one can apply the Poincar\'{e} inequality to get that
$$\|\theta\|_{L^2(\Omega)}\leq C^{*}\|\nabla\theta\|_{L^2(\Omega)},\,\,\|u\|_{L^2(\Omega)}\leq C^{*}\|\nabla u\|_{L^2(\Omega)},\,\,\|B\|_{L^2(\Omega)}\leq C^{*}\|\nabla B\|_{L^2(\Omega)}$$
for the constant $C^*$ depending only on $\Omega$.

Thus, we can update \eqref{L2} as
\beno
\f{d}{dt}\|\theta,\,u,\,B\|_{L^2(\Omega)}^2+\frac{2}{C_0C^*}\|\theta,\,u,\,B\|_{L^2(\Omega)}^2\leq 0,
\eeno
which yields, after applying the Gronwall's inequality, that
\ben\label{L21}
\|\theta(\cdot,t),\,u(\cdot,t),\,B(\cdot,t)\|_{L^2(\Omega)}^2\leq e^{-2\alpha t}\|\theta_0,\,u_0,\,B_0\|_{L^2(\Omega)}^2, \,\forall\,\,t\geq0,
\een
where $\alpha=(C_0C^*)^{-1}$.

Then we multiply $e^{\alpha t}$ on both sides of \eqref{L2} and employ \eqref{L21} to derive
\beno
&&\f{d}{dt}(e^{\alpha t}\|\theta,\,u,\,B\|_{L^2(\Omega)}^2)+2C_0^{-1}e^{\alpha t}\|\nabla\theta,\,\nabla u,\,\nabla B\|_{L^2(\Omega)}^2\\
&\leq&\alpha e^{\alpha t}\|\theta,\,u,\,B\|_{L^2(\Omega)}^2\leq \alpha e^{-\alpha t}\|\theta_0,\,u_0,\,B_0\|_{L^2(\Omega)}^2
\eeno
for any $t\geq0$, which also implies, after integrating in time over $[0,t]$, that
\beno
\int_0^te^{\alpha \tau}\|\nabla\theta,\,\nabla u,\,\nabla B\|_{L^2(\Omega)}^2d\tau\leq\frac{C_0}{2}\|\theta_0,\,u_0,\,B_0\|_{L^2(\Omega)}^2,\,\,\,\forall t\geq0.
\eeno
\endproof

\subsection{$L^\infty$ Estimates of Temperature}\hspace*{\parindent}

\begin{lemma}\label{L32}
Under the assumptions of Proposition \ref{estimates}, if in addition, $\theta_0\in L^p(\Omega)$, then there holds that
\beno
\|\theta(\cdot,t)\|_{L^p(\Omega)}\leq C
\eeno
for any $p\in[2,\infty]\,\,{\rm and}\,\,t\geq0$, where $C$ only depends on $\|\theta_0,\,u_0,\,B_0\|_{L^2(\Omega)},\,\,\,\|\theta_0\|_{L^p(\Omega)}$, $C_0$ and $C^*$.
\end{lemma}
\no{\bf Proof.}\quad
For any $2\leq p<\infty$, multiplying $\eqref{B-MHD-equ}^1$ with $|\theta|^{p-2}\theta$ and using H\"{o}lder inequality, it follows that
\beno
\f1p\f{d}{dt}\|\theta\|_{L^p(\Omega)}^p+\f{4(p-1)}{p^2C_0}\|\nabla|\theta|^{\frac p2}\|_{L^2(\Omega)}^2\leq \|u\|_{L^p(\Omega)}\|\theta\|_{L^p(\Omega)}^{p-1},
\eeno
which yields that
\ben\label{Lp}
\f{d}{dt}\|\theta\|_{L^p(\Omega)}\leq \|u\|_{L^p(\Omega)}.
\een

Now we integrate on both sides of \eqref{Lp} in time over $[0,t]$, make use of the Soblev embedding and Lemma \ref{L31} to get
\ben
\|\theta\|_{L^p(\Omega)}&\leq& \|\theta_0\|_{L^p(\Omega)}+\int_0^t\|u\|_{L^p(\Omega)}d\tau\notag\\
&\leq&\|\theta_0\|_{L^p(\Omega)}+\int_0^t\|u\|_{H^1(\Omega)}d\tau\notag\\
&\leq&\|\theta_0\|_{L^p(\Omega)}+\int_0^te^{-\f{\alpha\tau}2}e^{\f{\alpha\tau}2}\|u\|_{H^1(\Omega)}d\tau\label{Lp1}\\
&\leq&\|\theta_0\|_{L^p(\Omega)}+(\int_0^te^{-\alpha\tau}d\tau)^{\f12}(\int_0^te^{\alpha\tau}\|u\|_{H^1(\Omega)}^2d\tau)^{\f12}\notag\\
&\leq&C.\notag
\een
Because the constant $C$ in \eqref{Lp1} is independent of $p$, by letting $p\rightarrow\infty$, one can further derive that $\|\theta\|_{L^\infty(\Omega)}\leq C$.
\endproof

As an immediate consequence of Lemma \ref{L32} and the assumption that $\kappa(\theta),\,\mu(\theta),\,\sigma(\theta)$ are smooth, it holds that
\ben\label{M}
\max_{\theta\in[-\|\theta_0\|_{L^\infty},\,\|\theta_0\|_{L^\infty}]}\{|\kappa(\cdot),\kappa'(\cdot),\kappa''(\cdot)|,|\mu(\cdot),\mu'(\cdot),\mu''(\cdot)|,|\sigma(\cdot),\sigma'(\cdot),\sigma''(\cdot)|\}\leq\ M,\,\,\,\,\,
\een
where $M$ is a constant depending only on $\|\theta_0\|_{L^\infty}$.

On the basis of \eqref{M}, the definition of $\hat\theta$ \eqref{hatthetadef} and assumption \eqref{Assumption}, it is not hard to derive the following property.
\begin{lemma}\label{hattheta}
For $\hat{\theta}$ defined in $\eqref{hatthetadef}$ and arbitrary $p\in[2,\infty]$, it follows that
\ben
&&C_0^{-1}\|\nabla\theta\|_{L^p(\Omega)}\leq\|\nabla\hat{\theta}\|_{L^p(\Omega)}\leq\ M\|\nabla\theta\|_{L^p(\Omega)},\label{equiv1}\\
&&C_0^{-1}\|\theta_t\|_{L^p(\Omega)}\leq\|\hat{\theta}_t\|_{L^p(\Omega)}\leq\ M\|\theta_t\|_{L^p(\Omega)}.\label{equiv2}
\een
\end{lemma}

The second one is the relation between $\|\nabla\hat{\theta}\|_{L^2(\Omega)}$, $\|\nabla^2{\hat{\theta}}\|_{L^2(\Omega)}$ and $\|\nabla^2{\theta}\|_{L^2(\Omega)}$, which can be summarized as below.

\begin{proposition}\label{hattheta1}
For $\hat{\theta}$ defined in $\eqref{hatthetadef}$, there holds that
\ben
\|\nabla^2{\theta}\|_{L^2(\Omega)}\leq C\big(\|\nabla^2{\hat\theta}\|_{L^2(\Omega)}+\|\nabla{\hat\theta}\|_{L^2(\Omega)}^2\big),\label{equiv3}
\een
where $C$ depends on $C_0$ and $M$ only.
\end{proposition}
\no{\bf Proof.}\quad
Thanks to \eqref{hatthetadef}, one has $\p_i{\hat\theta}=\kappa(\theta)\p_i\theta$, which also implies, after direct calculation,
$$\p_i\p_j\theta=\kappa^{-1}(\theta)\p_i\p_j{\hat\theta}-\kappa^{-3}(\theta)\kappa'(\theta)\p_i{\hat\theta}\p_j{\hat\theta}.$$

Then by using \eqref{Assumption} and \eqref{M}, we have
\ben
\|\p_i\p_j{\theta}\|_{L^2(\Omega)}\leq C_0\|\p_i\p_j{\hat\theta}\|_{L^2(\Omega)}+C_0^3M\|\p_i{\hat\theta}\|_{L^2(\Omega)}\|\p_j{\hat\theta}\|_{L^2(\Omega)},\label{equiv4}
\een
which yields \eqref{equiv3} by summing over \eqref{equiv4} about $i$ and $j$.

\endproof

\subsection{$H^1$ Estimates}\hspace*{\parindent}

\begin{lemma}\label{L33}
Under the assumptions of Proposition \ref{estimates}, $\forall\,\,t\geq0$, there holds that
\beno
\|\nabla{\theta},\,\nabla\hat{\theta}\|_{L^2(\Omega)}^2\leq Ce^{-\alpha t}
\eeno
and
\beno
\int_0^te^{\alpha \tau}\|{\theta}_{\tau},\,{\hat\theta}_{\tau}\|_{L^2(\Omega)}^2d\tau+\int_0^te^{\alpha \tau}\|\Delta{\theta},\,\Delta{\hat\theta}\|_{L^2(\Omega)}^2d\tau\leq C,
\eeno
where $C$ depends on $\|\theta_0,\,u_0,\,B_0\|_{L^2(\Omega)}$, $\|\nabla{\theta_0}\|_{L^2(\Omega)}$, $C_0$, $M$ and $\alpha$.
\end{lemma}
\no{\bf Proof.}\quad
Multiplying \eqref{simtheta} with $-\Delta\hat{\theta}$, applying \eqref{ibsimtheta}, Corollary \ref{C1}, Lemma {\ref{elliptic}} and Lemma {\ref{L31}}, one can get
\beno
&&\f12\f{d}{dt}\|\nabla\hat{\theta}\|_{L^2(\Omega)}^2+C_0^{-1}\|\Delta{\hat\theta}\|_{L^2(\Omega)}^2\leq\int_{\Omega}u\cdot\nabla {\hat\theta}\Delta{\hat\theta}dx+\int_{\Omega}\kappa(\theta)u_2\Delta{\hat\theta}dx\\
&\leq&\|u\cdot\nabla{\hat\theta}\|_{L^2(\Omega)}\|\Delta{\hat\theta}\|_{L^2(\Omega)}+M\|u\|_{L^2(\Omega)}\|\Delta{\hat\theta}\|_{L^2(\Omega)}\\
&\leq&\f{C_0^{-1}}4\|\Delta{\hat\theta}\|_{L^2(\Omega)}^2+C\|u\cdot\nabla{\hat\theta}\|_{L^2(\Omega)}^2+C\|u\|_{L^2(\Omega)}^2\\
&\leq&\f{C_0^{-1}}4\|\Delta{\hat\theta}\|_{L^2(\Omega)}^2+C\|u\|_{L^2(\Omega)}\|\nabla u\|_{L^2(\Omega)}\|\nabla {\hat\theta}\|_{L^2(\Omega)}\|\nabla^2{\hat\theta}\|_{L^2(\Omega)}\\
&&+C\|u\|_{L^2(\Omega)}\|\nabla u\|_{L^2(\Omega)}\|\nabla {\hat\theta}\|_{L^2(\Omega)}^{2}+C\|u\|_{L^2(\Omega)}^2\\
&\leq&\f{C_0^{-1}}2\|\Delta{\hat\theta}\|_{L^2(\Omega)}^2+C\|\nabla u\|_{L^2(\Omega)}^2\|\nabla {\hat\theta}\|_{L^2(\Omega)}^2+C\|\nabla u\|_{L^2(\Omega)}\|\nabla {\hat\theta}\|_{L^2(\Omega)}^{2}+C\|u\|_{L^2(\Omega)}^2\\
&\leq&\f{C_0^{-1}}2\|\Delta{\hat\theta}\|_{L^2(\Omega)}^2
+C\|\nabla u\|_{L^2(\Omega)}^2\|\nabla {\hat\theta}\|_{L^2(\Omega)}^{2}+C\|u\|_{L^2(\Omega)}^2+C\|\nabla{\hat\theta}\|_{L^2(\Omega)}^2,
\eeno
which yields, after multiplying by $e^{\alpha t}$ on both sides of above inequality, that
\beno
&&\f{d}{dt}(e^{\alpha t}\|\nabla\hat{\theta}\|_{L^2(\Omega)}^2)+C_0^{-1}e^{\alpha t}\|\Delta{\hat\theta}\|_{L^2(\Omega)}^2\\
&\leq&Ce^{\alpha t}\|\nabla u\|_{L^2(\Omega)}^2\|\nabla {\hat\theta}\|_{L^2(\Omega)}^{2}+Ce^{\alpha t}\|u\|_{L^2(\Omega)}^2+Ce^{\alpha t}\|\nabla{\hat\theta}\|_{L^2(\Omega)}^2.
\eeno
This, together with Gronwall's inequality and \eqref{equiv1} shows
\ben
&&e^{\alpha t}\|\nabla{\theta},\,\nabla\hat{\theta}\|_{L^2(\Omega)}^2+C_0^{-1}\int_0^te^{\alpha \tau}\|\Delta{\hat\theta}\|_{L^2(\Omega)}^2d\tau\leq C\|\nabla{\theta_0}\|_{L^2(\Omega)}^2.\label{INE331}
\een

Then by employing \eqref{simtheta}, \eqref{equiv2}, \eqref{INE331}, repeating the same calculation as above and using Lemma \ref{L31} again, we have
\beno
&&\int_0^te^{\alpha \tau}\|{\theta}_{\tau}\|_{L^2(\Omega)}^2d\tau\leq C_0\int_0^te^{\alpha \tau}\|{\hat\theta}_{\tau}\|_{L^2(\Omega)}^2d\tau\\
&\leq& C\big[\int_0^te^{\alpha \tau}\|\Delta{\hat\theta}\|_{L^2(\Omega)}^2d\tau+\int_0^te^{\alpha \tau}\|\nabla u\|_{L^2(\Omega)}^2\|\nabla {\hat\theta}\|_{L^2(\Omega)}^{2}d\tau\\
&&+\int_0^te^{\alpha \tau}\|\nabla u\|_{L^2(\Omega)}\|\nabla {\hat\theta}\|_{L^2(\Omega)}^{2}d\tau+\int_0^te^{\alpha \tau}\|u\|_{L^2(\Omega)}^2d\tau\big]\\
&\leq& C\big[\int_0^te^{\alpha \tau}\|\Delta{\hat\theta}\|_{L^2(\Omega)}^2d\tau+\int_0^te^{\alpha \tau}\|\nabla u\|_{L^2(\Omega)}^2d\tau+\int_0^te^{-\alpha \tau}d\tau\big]\\
&\leq&C.
\eeno
Finally, thanks to Lemma \ref{elliptic}, Proposition \ref{hattheta1} and \eqref{INE331}, there holds that
\beno
&&\int_0^te^{\alpha \tau}\|\Delta{\theta}\|_{L^2(\Omega)}^2d\tau\leq C\big[\int_0^te^{\alpha \tau}\|\Delta{\hat\theta}\|_{L^2(\Omega)}^2d\tau+\int_0^te^{\alpha \tau}\|\nabla {\hat\theta}\|_{L^2(\Omega)}^{4}d\tau\big]\\
&\leq& C\big[\int_0^te^{\alpha \tau}\|\Delta{\hat\theta}\|_{L^2(\Omega)}^2d\tau+\int_0^te^{-\alpha \tau}d\tau\big]\\
&\leq&C.
\eeno
\endproof

\begin{lemma}\label{L34}
Under the assumptions of Proposition \ref{estimates}, $\forall\,\,t\geq0$, there holds that
\beno
\|\nabla{u},\,\nabla{B}\|_{L^2(\Omega)}^2\leq Ce^{-\alpha t}
\eeno
and
\beno
\int_0^te^{\alpha \tau}\|\Delta{u},\,\Delta{B}\|_{L^2(\Omega)}^2d\tau\leq C,
\eeno
where $C$ depends on $\|\theta_0,\,u_0,\,B_0\|_{H^1(\Omega)}$, $M$, $C_0$ and $\alpha$.
\end{lemma}
\no{\bf Proof.}\quad
On the basis of direct calculation, we can rewrite $\eqref{B-MHD-equ}^2$ and $\eqref{B-MHD-equ}^3$ as
\begin{equation}\label{B-MHD-equ-uB}
\begin{cases}
 \p_t u + u\cdot\nabla u -\mu(\theta)\Delta u+
\grad \Pi=\theta e_2+B\cdot\nabla B+\nabla\mu(\theta)\cdot \nabla u, \\
 \p_t B+u \cdot \grad B-\sigma(\theta)\Delta B=B\cdot \grad u+\grad^{\perp}\sigma(\theta)\grad^{\perp}\cdot B.
\end{cases}
\end{equation}

By taking inner product of $\eqref{B-MHD-equ-uB}^1$ with $-\Delta u$ and $\eqref{B-MHD-equ-uB}^2$ with $-\Delta B$, integrating by parts, one has
\beno
&&\f12\f{d}{dt}(\|\nabla u\|_{L^2(\Omega)}^2+\|\nabla B\|_{L^2(\Omega)}^2)+C_0^{-1}\|\Delta u\|_{L^2(\Omega)}^2+C_0^{-1}\|\Delta B\|_{L^2(\Omega)}^2\\
&\leq&-\int_{\Omega}\theta e_2\cdot\Delta udx-\int_{\Omega}B\cdot\nabla B\cdot\Delta udx-\int_{\Omega}u\cdot\nabla B\cdot\Delta Bdx-\int_{\Omega}B\cdot\nabla u\cdot\Delta Bdx\\
&&-\int_{\Omega}\nabla\mu(\theta)\cdot\nabla u\cdot\Delta{u}dx-\int_{\Omega}\grad^{\perp}\cdot B\grad^{\perp}\sigma(\theta)\cdot\Delta{B}dx\\
&=&\sum\limits_{j=1}^6I_j,
\eeno
where we use the fact that $\int_{\Omega}u\cdot\nabla u\cdot\Delta u dx=0$ by $u|_{\partial \Omega}=0$ and integrating by part. Subsequently, we will estimate the six terms one by one. Initially, by the H\"{o}lder inequality and Young inequality, it is clear that
\beno
I_1&\leq&\|\theta\|_{L^2(\Omega)}\|\Delta u\|_{L^2(\Omega)}\\
&\leq&\f{C_0^{-1}}6\|\Delta u\|_{L^2(\Omega)}^2+C\|\theta\|_{L^2(\Omega)}^2.
\eeno
Then by H\"{o}lder inequality, Corollary \ref{C1}, Lemma {\ref{elliptic}}, Lemma \ref{L31} and Young inequality, it follows that
\beno
&&I_2+I_3+I_4\\
&\leq&\|B\|_{L^4(\Omega)}\|\nabla B\|_{L^4(\Omega)}\|\Delta u\|_{L^2(\Omega)}+\|u\|_{L^4(\Omega)}\|\nabla B\|_{L^4(\Omega)}\|\Delta B\|_{L^2(\Omega)}\\
&&+\|B\|_{L^4(\Omega)}\|\nabla u\|_{L^4(\Omega)}\|\Delta B\|_{L^2(\Omega)}\\
&\leq&\|B\|_{L^2(\Omega)}^{\f12}\|\nabla B\|_{L^2(\Omega)}^{\f12}(\|\nabla B\|_{L^2(\Omega)}^{\f12}\|\Delta B\|_{L^2(\Omega)}^{\f12}+\|\nabla B\|_{L^2(\Omega)})\|\Delta u\|_{L^2(\Omega)}\\
&&+\|u\|_{L^2(\Omega)}^{\f12}\|\nabla u\|_{L^2(\Omega)}^{\f12}(\|\nabla B\|_{L^2(\Omega)}^{\f12}\|\Delta B\|_{L^2(\Omega)}^{\f12}+\|\nabla B\|_{L^2(\Omega)})\|\Delta B\|_{L^2(\Omega)}\\
&&+\|B\|_{L^2(\Omega)}^{\f12}\|\nabla B\|_{L^2(\Omega)}^{\f12}(\|\nabla u\|_{L^2(\Omega)}^{\f12}\|\Delta u\|_{L^2(\Omega)}^{\f12}+\|\nabla u\|_{L^2(\Omega)})\|\Delta B\|_{L^2(\Omega)}\\
&\leq&\|\nabla B\|_{L^2(\Omega)}\|\Delta B\|_{L^2(\Omega)}^{\f12}\|\Delta u\|_{L^2(\Omega)}+\|\nabla B\|_{L^2(\Omega)}^{\f32}\|\Delta u\|_{L^2(\Omega)}\\
&&+\|\nabla u\|_{L^2(\Omega)}^{\f12}\|\nabla B\|_{L^2(\Omega)}^{\f12}\|\Delta B\|_{L^2(\Omega)}^{\f32}+\|\nabla u\|_{L^2(\Omega)}^{\f12}\|\nabla B\|_{L^2(\Omega)}\|\Delta B\|_{L^2(\Omega)}\\
&&+\|\nabla B\|_{L^2(\Omega)}^{\f12}\|\nabla u\|_{L^2(\Omega)}^{\f12}\|\Delta u\|_{L^2(\Omega)}^{\f12}\|\Delta B\|_{L^2(\Omega)}+\|\nabla B\|_{L^2(\Omega)}^{\f12}\|\nabla u\|_{L^2(\Omega)}\|\Delta B\|_{L^2(\Omega)}\\
&\leq&\f{C_0^{-1}}6(\|\Delta u\|_{L^2(\Omega)}^2+\|\Delta B\|_{L^2(\Omega)}^2)+C(\|\nabla u\|_{L^2(\Omega)}^2+\|\nabla B\|_{L^2(\Omega)}^2)^2\\
&&+C(\|\nabla u\|_{L^2(\Omega)}^2+\|\nabla B\|_{L^2(\Omega)}^2).
\eeno
Regarding the last two terms, thanks to \eqref{M}, H\"{o}lder inequality, Corollary \ref{C1}, Lemma {\ref{elliptic}}, Lemma \ref{L31} and Young inequality, we have
\beno
&&I_5+I_6\\
&\leq&M\|\nabla\theta\|_{L^4(\Omega)}\big[\|\nabla u\|_{L^4(\Omega)}\|\Delta u\|_{L^2(\Omega)}+\|\nabla B\|_{L^4(\Omega)}\|\Delta B\|_{L^2(\Omega)}\big]\\
&\leq&M\|\nabla\theta\|_{L^2(\Omega)}^{\f12}\|\Delta \theta\|_{L^2(\Omega)}^{\f12}(\|\nabla u\|_{L^2(\Omega)}^{\f12}\|\Delta u\|_{L^2(\Omega)}^{\f12}+\|\nabla u\|_{L^2(\Omega)})\|\Delta u\|_{L^2(\Omega)}\\
&&+M\|\nabla\theta\|_{L^2(\Omega)}(\|\nabla u\|_{L^2(\Omega)}^{\f12}\|\Delta u\|_{L^2(\Omega)}^{\f12}+\|\nabla u\|_{L^2(\Omega)})\|\Delta u\|_{L^2(\Omega)}\\
&&+M\|\nabla\theta\|_{L^2(\Omega)}^{\f12}\|\Delta \theta\|_{L^2(\Omega)}^{\f12}(\|\nabla B\|_{L^2(\Omega)}^{\f12}\|\Delta B\|_{L^2(\Omega)}^{\f12}+\|\nabla u\|_{L^2(\Omega)})\|\Delta B\|_{L^2(\Omega)}\\
&&+M\|\nabla\theta\|_{L^2(\Omega)}(\|\nabla B\|_{L^2(\Omega)}^{\f12}\|\Delta B\|_{L^2(\Omega)}^{\f12}+\|\nabla B\|_{L^2(\Omega)})\|\Delta B\|_{L^2(\Omega)}\\
&\leq&C\|\Delta \theta\|_{L^2(\Omega)}^{\f12}\|\nabla u\|_{L^2(\Omega)}^{\f12}\|\Delta u\|_{L^2(\Omega)}^{\f32}+\|\Delta \theta\|_{L^2(\Omega)}^{\f12}\|\nabla u\|_{L^2(\Omega)}\|\Delta u\|_{L^2(\Omega)}\\
&&+C\|\nabla u\|_{L^2(\Omega)}^{\f12}\|\Delta u\|_{L^2(\Omega)}^{\f32}+C\|\nabla u\|_{L^2(\Omega)}\|\Delta u\|_{L^2(\Omega)}\\
&&+C\|\Delta \theta\|_{L^2(\Omega)}^{\f12}\|\nabla B\|_{L^2(\Omega)}^{\f12}\|\Delta B\|_{L^2(\Omega)}^{\f32}+\|\Delta \theta\|_{L^2(\Omega)}^{\f12}\|\nabla B\|_{L^2(\Omega)}\|\Delta B\|_{L^2(\Omega)}\\
&&+C\|\nabla B\|_{L^2(\Omega)}^{\f12}\|\Delta B\|_{L^2(\Omega)}^{\f32}+C\|\nabla B\|_{L^2(\Omega)}\|\Delta B\|_{L^2(\Omega)}\\
&\leq&\f{C_0^{-1}}6(\|\Delta u\|_{L^2(\Omega)}^2+\|\Delta B\|_{L^2(\Omega)}^2)+C\|\Delta \theta\|_{L^2(\Omega)}^2(\|\nabla u\|_{L^2(\Omega)}^2+\|\nabla B\|_{L^2(\Omega)}^2)\\
&&+C(\|\nabla u\|_{L^2(\Omega)}^2+\|\nabla B\|_{L^2(\Omega)}^2+\|\Delta \theta\|_{L^2(\Omega)}^2).
\eeno
Thus, summing up all the above inequalities, it yields that
\ben\label{INE332}
&&\f{d}{dt}(\|\nabla u\|_{L^2(\Omega)}^2+\|\nabla B\|_{L^2(\Omega)}^2)+C_0^{-1}\|\Delta u\|_{L^2(\Omega)}^2+C_0^{-1}\|\Delta B\|_{L^2(\Omega)}^2\notag\\
&\leq&C(\|\nabla u\|_{L^2(\Omega)}^2+\|\nabla B\|_{L^2(\Omega)}^2)^2+C\|\Delta \theta\|_{L^2(\Omega)}^2(\|\nabla u\|_{L^2(\Omega)}^2+\|\nabla B\|_{L^2(\Omega)}^2)\notag\\
&&+C(\|\theta\|_{L^2(\Omega)}^2+\|\nabla u\|_{L^2(\Omega)}^2+\|\nabla B\|_{L^2(\Omega)}^2+\|\Delta \theta\|_{L^2(\Omega)}^2),
\een
which also implies, after using Gronwall's inequality, Lemma \ref{L31} and \ref{L33}, that
\beno
&&\|\nabla u\|_{L^2(\Omega)}^2+\|\nabla B\|_{L^2(\Omega)}^2+C_0^{-1}\int_0^t(\|\Delta u\|_{L^2(\Omega)}^2+\|\Delta B\|_{L^2(\Omega)}^2)d\tau\\
&\leq&C{\rm exp}\big[{\int_0^t(\|\nabla u\|_{L^2(\Omega)}^2+\|\nabla B\|_{L^2(\Omega)}^2+\|\Delta \theta\|_{L^2(\Omega)}^2)d\tau}\big]\\
&&\times\big[\|\nabla u_0\|_{L^2(\Omega)}^2+\|\nabla B_0\|_{L^2(\Omega)}^2+\int_0^t\|\theta,\,\nabla u,\,\nabla B,\,\Delta \theta\|_{L^2(\Omega)}^2d\tau\big]\\
&\leq&C.
\eeno

Now, we can update \eqref{INE332} as
\ben\label{INE333}
&&\f{d}{dt}(\|\nabla u\|_{L^2(\Omega)}^2+\|\nabla B\|_{L^2(\Omega)}^2)+C_0^{-1}\|\Delta u\|_{L^2(\Omega)}^2+C_0^{-1}\|\Delta B\|_{L^2(\Omega)}^2\notag\\
&\leq&C(\|\theta\|_{L^2(\Omega)}^2+\|\nabla u\|_{L^2(\Omega)}^2+\|\nabla B\|_{L^2(\Omega)}^2+\|\Delta \theta\|_{L^2(\Omega)}^2),
\een
which deduces, after multiplying by $e^{\alpha t}$ on both sides of \eqref{INE333} and integrating in time over $[0,t]$, that
\beno
&&e^{\alpha t}\|\nabla u,\,\nabla B\|_{L^2(\Omega)}^2+C_0^{-1}\int_0^te^{\alpha \tau}\|\Delta{u},\,\Delta{B}\|_{L^2(\Omega)}^2d\tau\leq C\|\nabla{u_0},\,\nabla{B_0}\|_{L^2(\Omega)}^2.
\eeno
\endproof

\begin{corollary}\label{thetatesti}
Under the assumptions of Proposition \ref{estimates}, $\forall\,\,t\geq0$, there holds that
\beno
\int_0^te^{\alpha \tau}\|{u}_{\tau},\,{B}_{\tau}\|_{L^2(\Omega)}^2d\tau\leq C,
\eeno
where $C$ depends on $\|\theta_0,\,u_0,\,B_0\|_{H^1(\Omega)}$, $M$, $C_0$ and $\alpha$.
\end{corollary}
\no{\bf Proof.}\quad
Taking inner product of $\eqref{B-MHD-equ-uB}^1$ with $u_t$ and $\eqref{B-MHD-equ-uB}^2$ with $B_t$, integrating by parts, employing H\"{o}lder inequality Corollary \ref{C1}, Lemma \ref{elliptic}, Lemma \ref{L31}, Lemma \ref{L33} and Lemma \ref{L34}, one has
\beno
&&\|u_t\|_{L^2(\Omega)}^2+\|B_t\|_{L^2(\Omega)}^2\\
&\leq&\|\theta\|_{L^2(\Omega)}^2+M\|\Delta u\|_{L^2(\Omega)}^2+M\|\Delta B\|_{L^2(\Omega)}^2+\|u\cdot\nabla u\|_{L^2(\Omega)}^2+\|B\cdot\nabla B\|_{L^2(\Omega)}^2\\
&&+\|u\cdot\nabla B\|_{L^2(\Omega)}^2+\|B\cdot\nabla u\|_{L^2(\Omega)}^2+\|\nabla \theta\|_{L^4(\Omega)}^2(\|\nabla u\|_{L^4(\Omega)}^2+\|\nabla B\|_{L^4(\Omega)}^2)\\
&\leq&C\|\theta,\,\Delta u,\,\Delta B\|_{L^2(\Omega)}^2+(\|u\|_{L^2(\Omega)}\|\nabla u\|_{L^2(\Omega)}+\|B\|_{L^2(\Omega)}\|\nabla B\|_{L^2(\Omega)}\\
&&+\|\nabla \theta\|_{L^2(\Omega)}^2+\|\nabla \theta\|_{L^2(\Omega)}\|\Delta \theta\|_{L^2(\Omega)})\times(\|\nabla u\|_{L^2(\Omega)}\|\Delta u\|_{L^2(\Omega)}+\|\nabla u\|_{L^2(\Omega)}^2\\
&&+\|\nabla B\|_{L^2(\Omega)}^2+\|\nabla B\|_{L^2(\Omega)}\|\Delta B\|_{L^2(\Omega)})\\
&\leq&Ce^{-2\alpha t}+C\|\Delta\theta,\,\Delta u,\,\Delta B\|_{L^2(\Omega)}^2.
\eeno
On the basis of this inequality, we can multiply by $e^{\alpha t}$ on both sides of it and integrate in time over $[0,t]$ to get the conclusion.
\endproof

\begin{proof}[Proof of Theorem \ref{T1}]
The proof is a consequence of Schauder's fixed point theorem.
We shall only provide the sketches.

\vskip .1in
To define the functional setting, we fix $T>0$ and $R_0$ to
be specified later. For notational convenience, we
write
$$
X\equiv C(0,T; H^1_0(\Omega))\cap L^2(0,T; H^2(\Omega))
$$
with $\|g\|_X\equiv \|g\|_{C(0,T; H^1_0(\Omega))}+\|g\|_{L^2(0,T; H^2(\Omega))}$,
and define
$$
D=\{g\in X\,|\,\|g\|_X\leq R_0\}.
$$
Clearly, $D\subset X$ is closed and convex.

\vskip .1in
We fix $\epsilon\in(0, 1)$ and define a continuous map on $D$. For any $f,\,g\in D$, we regularize it and the initial data $(\theta_0,u_0,B_0)$ via the
standard mollifying process,
$$
g^{\epsilon}= \rho^{\epsilon}\ast g,\quad\theta_{0}^{\epsilon} = \rho^{\epsilon}\ast \theta_0,\quad u_{0}^{\epsilon} = \rho^{\epsilon}\ast u_0,\quad B_{0}^{\epsilon} = \rho^{\epsilon}\ast B_0,
$$
where $\rho^{\epsilon}$ is the standard mollifier. According to
Lemma \ref{Boussinesq}, the 2D Boussinesq system with smooth external forcing $g^{\epsilon}\cdot \nabla g^{\epsilon}$ and smooth initial data $u_0^{\epsilon},\theta_0^{\epsilon}$
\begin{equation}\label{weq1}
\left\{\begin{array}{ll}
u_t+u\cdot\nabla u-\nabla\cdot(\mu(\theta)\nabla u)+\nabla p=\theta e_2+g^{\epsilon}\cdot \nabla g^{\epsilon},\\
\theta_t+u\cdot\nabla\theta-\nabla\cdot(\kappa(\theta)\nabla \theta)=-u_2,\\
\nabla\cdot u=0,\\
(u,\theta)|_{t=0}=(u_0,\theta_0)(x),~~(u,\theta)|_{\partial \Omega}=0.
\end{array}\right.
\end{equation}
has a unique solution $u^{\epsilon},\,\theta^{\epsilon}$. We then solve the linear parabolic equation with the smooth initial data $B_0^{\epsilon}$
\begin{equation}\label{weq2}
\left\{\begin{array}{ll}
 \p_t B+u^{\epsilon} \cdot \grad B-\grad^{\perp}(\sigma(\theta^{\epsilon})\grad^{\perp}\cdot B)=B\cdot \grad u^{\epsilon},\\
B(x,0)=B_0^{\epsilon}(x),\quad B|_{\p \Omega}=0,
\end{array}\right.
\end{equation}
and denote the solution by $B^{\epsilon}$. This process
allows us to define the map
$$
F^{\epsilon}(g)=w^{\epsilon}.
$$

We then apply Schauder's fixed point theorem to construct a sequence of approximate solutions to \eqref{B-MHD-equ}-\eqref{B-MHD-boundary-B}. It suffices to show that, for any fixed $\epsilon\in(0, 1)$, $F^{\epsilon}: D\rightarrow D$
is continuous and compact. More precisely, we need to show
\begin{enumerate}
\item[(a)] $\|B^{\epsilon}\|_{D}\leq R_0$;
\item[(b)] $\|F^{\epsilon}(g_1)-F^{\epsilon}(g_2)\|_{D}\leq C\|g_1-g_2\|_{D}$ for $C$ indepedent of $\epsilon$ and any $g_1,\,g_2\in D$.
\end{enumerate}

These estimates can be verified as in the proof of Lemma \ref{L31}-\ref{L34}
and we omit the details. In addition, as in the proof of Lemma \ref{L31}-\ref{L34},
we can show that
$$
\|B^\epsilon\|_{C(0,T;H_0^1(\Omega))} + \|B^\epsilon\|_{L^2(0,T;H^2(\Omega))}\leq C,
$$
for a constant $C$ independent of $\epsilon$. These uniform estimates would allow us to pass the limit to obtain a weak solution $(u,\,\theta,\,B)$ as
stated in Theorem \ref{T1}. This completes the proof.
\end{proof}

\section{Global strong solution}\hspace*{\parindent}
In the section, we will concentrate on deriving the global strong solution, i.e the proof of Theorem \ref{T2}. As described in the introduction, to prove Theorem \ref{T2}, the first step is the desired $H^2(\Omega)$ estimates, which is summarized by Lemma \ref{L35} and \ref{L36} as below.

\subsection{$H^2$ Estimates}\hspace*{\parindent}
Due to the appear of boundary effects, we will make use of the estimates of time derivatives and Lemma \ref{elliptic}-\ref{stokes2} to obtain the estimates of spatial derivatives. Therefore, the main work is the $L^2$ estimates of time derivatives.

\begin{lemma}\label{L35}
Suppose $\Omega$ be a bounded domain with smooth boundary, $(u_0,\,B_0)\in H_0^{1}(\Omega)$, $\theta_0\in H^{2}(\Omega)$, and $(\theta,\,u,\,B)$ solves the system \eqref{B-MHD-equ}-\eqref{B-MHD-boundary-B}, then $\forall\,\,t\geq0$, it holds that
\beno
\|\nabla^2{\theta},\,{\theta}_t\|_{L^2(\Omega)}^2\leq Ce^{-\alpha t}
\eeno
and
\beno
\int_0^te^{\alpha \tau}\|\nabla{\theta}_{\tau}\|_{L^2(\Omega)}^2d\tau+\int_0^te^{\alpha \tau}\|\theta\|_{H^3(\Omega)}^2d\tau\leq C,
\eeno
where $C$ depends on $\|\theta_0\|_{H^2(\Omega)}$, $\|u_0,\,B_0\|_{H^1(\Omega)}$, $C_0$, $M$ and $\alpha$.
\end{lemma}
\no{\bf Proof.}\quad
Taking the temporal derivative of $\eqref{B-MHD-equ}^1$, it follows that
\ben\label{thetatequ}
\p_{t}{\theta_t}+u\cdot\nabla{\theta_t}+u_t\cdot\nabla \theta-\nabla\cdot(\kappa(\theta)\nabla{\theta_t})=\nabla\cdot(\kappa'(\theta){\theta_t}\nabla{\theta})-\p_tu_2.
\een
Then multiplying \eqref{thetatequ} by $\theta_t$, integrating on $\Omega$, applying Corollary \ref{C1}, Lemma \ref{elliptic} and Young inequality, we can obtain that
\ben\label{equ34}
&&\f12\f{d}{dt}\|\theta_t\|_{L^2(\Omega)}^2+C_0^{-1}\|\nabla \theta_t\|_{L^2(\Omega)}^2\notag\\
&\leq&-\int_{\Omega}u_t\cdot\nabla\theta\theta_tdx-\int_{\Omega}\kappa'(\theta)\theta_t\nabla\theta\cdot\nabla\theta_tdx-\int_{\Omega}\p_tu_2\theta_tdx\notag\\
&\leq&\|u_t\|_{L^2(\Omega)}\|\theta_t\|_{L^4(\Omega)}\|\nabla \theta\|_{L^4(\Omega)}+M\|\theta_t\|_{L^4(\Omega)}\|\nabla \theta\|_{L^4(\Omega)}\|\nabla \theta_t\|_{L^2(\Omega)}\notag\\
&&+\|u_t\|_{L^2(\Omega)}\|\theta_t\|_{L^2(\Omega)}\notag\\
&\leq&C\|u_t\|_{L^2(\Omega)}\|\theta_t\|_{L^2(\Omega)}^{\f12}\|\nabla\theta_t\|_{L^2(\Omega)}^{\f12}\|\nabla \theta\|_{L^2(\Omega)}^{\f12}\|\Delta \theta\|_{L^2(\Omega)}^{\f12}\notag\\
&&+C\|u_t\|_{L^2(\Omega)}\|\theta_t\|_{L^2(\Omega)}^{\f12}\|\nabla\theta_t\|_{L^2(\Omega)}^{\f12}\|\nabla \theta\|_{L^2(\Omega)}\\
&&+C\|\theta_t\|_{L^2(\Omega)}^{\f12}\|\nabla \theta_t\|_{L^2(\Omega)}^{\f32}(\|\nabla \theta\|_{L^2(\Omega)}^{\f12}\|\Delta \theta\|_{L^2(\Omega)}^{\f12}+\|\nabla \theta\|_{L^2(\Omega)})\notag\\
&&+\|u_t\|_{L^2(\Omega)}\|\theta_t\|_{L^2(\Omega)}\notag\\
&\leq&\f{C_0^{-1}}2\|\nabla\theta_t\|_{L^2(\Omega)}^2+C(\|\nabla \theta\|_{L^2(\Omega)}^2\|\Delta \theta\|_{L^2(\Omega)}^2+\|\nabla \theta\|_{L^2(\Omega)}^4)\|\theta_t\|_{L^2(\Omega)}^2\notag\\
&&+C\|u_t\|_{L^2(\Omega)}^2+C\|\theta_t\|_{L^2(\Omega)}^2,\notag
\een
which further yields, after employing Gronwall's inequality, Lemma \ref{L33} and Corollary \ref{thetatesti}, that
\beno
&&\|\theta_t\|_{L^2(\Omega)}^2+C_0^{-1}\int_0^t\|\nabla \theta_{\tau}\|_{L^2(\Omega)}^2d\tau\\
&\leq&C{\rm exp}\big[{\int_0^t(\|\nabla \theta\|_{L^2(\Omega)}^2\|\Delta \theta\|_{L^2(\Omega)}^2+\|\nabla \theta\|_{L^2(\Omega)}^4)d\tau}\big]\\
&&\times\big[\|\theta_t(0,\,x)\|_{L^2(\Omega)}^2+\int_0^t\|\theta_\tau,\,u_\tau\|_{L^2(\Omega)}^2d\tau\big]\\
&\leq&C.
\eeno
In fact, from $\eqref{B-MHD-equ}^1$, the value of $\|\theta_t(0,\,x)\|_{L^2(\Omega)}$ can be controlled by $\|\theta_0\|_{H^2(\Omega)}$. Now, we can update \eqref{equ34} as
\ben\label{equ341}
&&\f{d}{dt}\|\theta_t\|_{L^2(\Omega)}^2+C_0^{-1}\|\nabla \theta_t\|_{L^2(\Omega)}^2\\
&\leq&C(\|\nabla \theta\|_{L^2(\Omega)}^2\|\Delta \theta\|_{L^2(\Omega)}^2+\|\nabla \theta\|_{L^2(\Omega)}^4)+C\|u_t\|_{L^2(\Omega)}^2+C\|\theta_t\|_{L^2(\Omega)}^2,\notag
\een
which implies, after multiplying by $e^{\alpha t}$ on both sides of \eqref{equ341} and integrating in time over $[0,t]$, that
\ben\label{equ342}
&&e^{\alpha t}\|\theta_t\|_{L^2(\Omega)}^2+C_0^{-1}\int_0^te^{\alpha \tau}\|\nabla\theta_{\tau}\|_{L^2(\Omega)}^2d\tau\leq C.
\een

Subsequently, we get to do the $H^2(\Omega)$ estimates. By using \eqref{hatthetadef}, \eqref{Assumption}, Lemma \ref{hattheta} and Proposition \ref{hattheta1}, it yields that
\beno
&&\|\nabla^2\theta\|_{L^2(\Omega)}^2\leq C\|\Delta{\hat\theta}\|_{L^2(\Omega)}^2+C\|\nabla{\hat\theta}\|_{L^2(\Omega)}^4\\
&\leq&C\|\theta_t\|_{L^2(\Omega)}^2+C\|u\cdot\nabla{\hat\theta}\|_{L^2(\Omega)}^2+C\|u\|_{L^2(\Omega)}^2+C\|\nabla{\theta}\|_{L^2(\Omega)}^4\\
&\leq&Ce^{-\alpha t}+C\|u\|_{L^4(\Omega)}^2\|\nabla{\theta}\|_{L^4(\Omega)}^2\\
&\leq&Ce^{-\alpha t}+C\|u\|_{L^2(\Omega)}\|\nabla u\|_{L^2(\Omega)}(\|\nabla{\theta}\|_{L^2(\Omega)}^2+\|\nabla{\theta}\|_{L^2(\Omega)}\|\nabla^2{\theta}\|_{L^2(\Omega)})\\
&\leq&\f12\|\nabla^2{\theta}\|_{L^2(\Omega)}^2+Ce^{-\alpha t}+C\|u\|_{L^2(\Omega)}^2\|\nabla u\|_{L^2(\Omega)}^2\|\nabla{\theta}\|_{L^2(\Omega)}^2\\
&&+C\|u\|_{L^2(\Omega)}\|\nabla u\|_{L^2(\Omega)}\|\nabla{\theta}\|_{L^2(\Omega)}^2\\
&\leq&\f12\|\nabla^2{\theta}\|_{L^2(\Omega)}^2+Ce^{-\alpha t}.
\eeno
This clearly implies
\ben\label{equ343}
\|\nabla^2\theta\|_{L^2(\Omega)}^2\leq&Ce^{-\alpha t}.
\een

Recall the equation $\eqref{B-MHD-equ}^1$ satisfied by $\theta$, we can rewrite it as
$$-\kappa(\theta)\Delta\theta=\kappa'(\theta)\nabla\theta\cdot\nabla\theta-\theta_t-u\cdot\nabla\theta-u_2.$$
Considering that $\kappa(\theta)$ is positive and bounded by $C_0^{-1}$ below, by Corollary \ref{elliptic}, Lemma \ref{L31}-\ref{L35}, \eqref{equ342}-\eqref{equ343}, there holds that
\beno
&&\|\theta\|_{H^3(\Omega)}\\
&\leq&C\|\kappa'(\theta)\nabla\theta\cdot\nabla{\theta}\|_{H^1(\Omega)}+C\|\theta_t\|_{H^1(\Omega)}+C\|u\cdot\nabla{\theta}\|_{H^1(\Omega)}+C\|u\|_{H^1(\Omega)}\\
&\leq&C\|\nabla\theta\|_{L^4(\Omega)}^2+C\|\nabla\theta\|_{L^4(\Omega)}^2\|\nabla\theta\|_{L^2(\Omega)}+C\|\nabla^2\theta\|_{L^4(\Omega)}\|\nabla\theta\|_{L^{\f43}(\Omega)}
\\
&&+C\|\theta_t\|_{H^1(\Omega)}+C\|u\|_{L^4(\Omega)}\|\nabla\theta\|_{L^4(\Omega)}+C\|\nabla u\|_{L^4(\Omega)}\|\nabla\theta\|_{L^4(\Omega)}\\
&&+C\|u\|_{L^4(\Omega)}\|\nabla^2\theta\|_{L^4(\Omega)}+C\|u\|_{H^1(\Omega)}\\
&\leq&Ce^{-\alpha t}+C\|\theta_t\|_{H^1(\Omega)}+Ce^{-\alpha t}\|\nabla^2\theta\|_{L^4(\Omega)}+Ce^{-\alpha t}\|\nabla u\|_{L^4(\Omega)}\\
&\leq&Ce^{-\alpha t}+C\|\theta_t\|_{H^1(\Omega)}+Ce^{-\alpha t}(\|\nabla^2{\theta}\|_{L^2(\Omega)}+\|\nabla^2{\theta}\|_{L^2(\Omega)}^{\f12}\|\nabla^3{\theta}\|_{L^2(\Omega)}^{\f12})\\
&&+Ce^{-\alpha t}(\|\nabla u\|_{L^2(\Omega)}+\|\nabla{u}\|_{L^2(\Omega)}^{\f12}\|\Delta{u}\|_{L^2(\Omega)}^{\f12})\\
&\leq&\f12\|\nabla^3{\theta}\|_{L^2(\Omega)}^2+Ce^{-\alpha t}+C\|\Delta u\|_{L^2(\Omega)}^2,
\eeno
which implies, after simple calculation, that $\int_0^te^{\alpha \tau}\|\theta\|_{H^3(\Omega)}^2d\tau\leq C$.
\endproof

\begin{lemma}\label{L36}
Under the assumptions of Lemma \ref{L35}, we further assume $(u_0,\,B_0)\in H^{2}(\Omega)$, then we have
\beno
\|\nabla^2u,\,\nabla^2B,\,u_t,\,B_t\|_{L^2(\Omega)}^2\leq Ce^{-\alpha t}
\eeno
and
\beno
\int_0^te^{\alpha \tau}\|\nabla{u}_{\tau},\,\nabla{B}_{\tau}\|_{L^2(\Omega)}^2d\tau+\int_0^te^{\alpha \tau}\|{u},\,{B}\|_{H^3(\Omega)}^2d\tau\leq C
\eeno
for any $t\geq0$, where $C$ depends on $\|\theta_0,\,u_0,\,B_0\|_{H^2(\Omega)}$, $C_0$, $M$ and $\alpha$.
\end{lemma}
\no{\bf Proof.}\quad
Taking the temporal derivative of $\eqref{B-MHD-equ}^2$ and $\eqref{B-MHD-equ}^3$, it follows that
\begin{equation}\label{uBtequ}
\begin{cases}
 \p_t u_t + u\cdot\nabla u_t -\nabla\cdot(\mu(\theta)
\nabla u_t)+\grad \Pi_t=\nabla\cdot(\mu'(\theta)\theta_t
\nabla u)-u_t\cdot\nabla u\\
+\theta_t e_2+B_t\cdot\nabla B+B\cdot\nabla B_t,\\
 \p_t B_t+u \cdot \grad B_t-\grad^{\perp}(\sigma(\theta)\grad^{\perp}\cdot B_t)=\grad^{\perp}(\sigma'(\theta)\theta_t\grad^{\perp}\cdot B_t)-u_t\cdot\nabla B\\
+B_t\cdot \grad u+B\cdot \grad u_t.
\end{cases}
\end{equation}
Now we take inner product of $\eqref{uBtequ}^1$ with $u_t$, $\eqref{uBtequ}^2$ with $B_t$ and apply Lemma \ref{curlB} to get that
\ben\label{equ351}
&&\f12\f{d}{dt}(\|u_t\|_{L^2(\Omega)}^2+\|B_t\|_{L^2(\Omega)}^2)+C_0^{-1}\|\nabla u_t\|_{L^2(\Omega)}^2+C_0^{-1}\|\nabla B_t\|_{L^2(\Omega)}^2\notag\\
&\leq&-\int_{\Omega}\mu'(\theta)\theta_t\nabla u\cdot\nabla u_tdx-\int_{\Omega}u_t\cdot\nabla u\cdot u_tdx+\int_{\Omega}B_t\cdot\nabla B\cdot u_tdx+\int_{\Omega}\theta_te_2u_tdx\notag\\
&&-\int_{\Omega}\sigma'(\theta)\theta_t\grad^{\perp}\cdot B\,\grad^{\perp}\cdot B_tdx-\int_{\Omega}u_t\cdot\nabla B\cdot B_tdx+\int_{\Omega}B_t\cdot\nabla u\cdot B_tdx\notag\\
&=&\sum\limits_{j=1}^7I_j,
\een
here we have used the fact $\int_{\Omega}B\cdot \nabla B_t\cdot u_tdx+\int_{\Omega}B\cdot \nabla u_t\cdot B_tdx=0$. Thanks to H\"{o}lder inequality, Corollary \ref{C1} and Young inequality, there holds that
\ben\label{equ352}
&&I_1+I_5\notag\\
&\leq&C\|\theta_t\|_{L^4(\Omega)}(\|\nabla u\|_{L^4(\Omega)}\|\nabla u_t\|_{L^2(\Omega)}+\|\nabla B\|_{L^4(\Omega)}\|\nabla B_t\|_{L^2(\Omega)})\notag\\
&\leq&\f{C_0^{-1}}4(\|\nabla u_t\|_{L^2(\Omega)}^2+\|\nabla B_t\|_{L^2(\Omega)}^2)+C\|\theta_t\|_{L^4(\Omega)}^2(\|\nabla u\|_{L^4(\Omega)}^2+\|\nabla B\|_{L^4(\Omega)}^2)\notag\\
&\leq&\f{C_0^{-1}}4(\|\nabla u_t\|_{L^2(\Omega)}^2+\|\nabla B_t\|_{L^2(\Omega)}^2)+C\|\theta_t\|_{L^2(\Omega)}\|\nabla\theta_t\|_{L^2(\Omega)}(\|\nabla u\|_{L^2(\Omega)}^2\notag\\
&&+\|\nabla B\|_{L^2(\Omega)}^2+\|\nabla u\|_{L^2(\Omega)}\|\Delta u\|_{L^2(\Omega)}+\|\nabla B\|_{L^2(\Omega)}\|\Delta B\|_{L^2(\Omega)})\\
&\leq&\f{C_0^{-1}}4(\|\nabla u_t\|_{L^2(\Omega)}^2+\|\nabla B_t\|_{L^2(\Omega)}^2)+C\|\theta_t\|_{L^2(\Omega)}^2\|\nabla\theta_t\|_{L^2(\Omega)}^2+C\|\nabla u\|_{L^2(\Omega)}^4\notag\\
&&+C\|\nabla B\|_{L^2(\Omega)}^4+C\|\nabla u\|_{L^2(\Omega)}^2\|\Delta u\|_{L^2(\Omega)}^2+\|\nabla B\|_{L^2(\Omega)}^2\|\Delta B\|_{L^2(\Omega)}^2.\notag
\een
For the left terms, similarly, one has
\beno
&&I_2+I_3+I_4+I_6+I_7\\
&\leq&C(\|u_t\|_{L^4(\Omega)}+\|B_t\|_{L^4(\Omega)})^2(\|\nabla u\|_{L^2(\Omega)}+\|\nabla B\|_{L^2(\Omega)})+\|\theta_t\|_{L^2(\Omega)}\| u_t\|_{L^2(\Omega)}\\
&\leq&C(\|u_t\|_{L^2(\Omega)}\|\nabla u_t\|_{L^2(\Omega)}+\|B_t\|_{L^2(\Omega)}\|\nabla B_t\|_{L^2(\Omega)})(\|\nabla u\|_{L^2(\Omega)}+\|\nabla B\|_{L^2(\Omega)})\\
&&+\|\theta_t\|_{L^2(\Omega)}\| u_t\|_{L^2(\Omega)}\\
&\leq&\f{C_0^{-1}}4(\|\nabla u_t\|_{L^2(\Omega)}^2+\|\nabla B_t\|_{L^2(\Omega)}^2)+C(\|\nabla u\|_{L^2(\Omega)}^2+\|\nabla B\|_{L^2(\Omega)}^2)(\|u_t\|_{L^2(\Omega)}^2\\
&&+\|B_t\|_{L^2(\Omega)}^2)+C\|\theta_t\|_{L^2(\Omega)}^2+C\| u_t\|_{L^2(\Omega)}^2,
\eeno
which together with \eqref{equ351} and \eqref{equ352} yield
\ben\label{equ353}
&&\f{d}{dt}(\|u_t\|_{L^2(\Omega)}^2+\|B_t\|_{L^2(\Omega)}^2)+C_0^{-1}\|\nabla u_t\|_{L^2(\Omega)}^2+C_0^{-1}\|\nabla B_t\|_{L^2(\Omega)}^2\notag\\
&\leq&C(\|\nabla u\|_{L^2(\Omega)}^2+\|\nabla B\|_{L^2(\Omega)}^2)(\|u_t\|_{L^2(\Omega)}^2+\|B_t\|_{L^2(\Omega)}^2)\notag\\
&&+C\|\theta_t\|_{L^2(\Omega)}^2\|\nabla\theta_t\|_{L^2(\Omega)}^2+C\|\nabla u\|_{L^2(\Omega)}^4+C\|\nabla B\|_{L^2(\Omega)}^4\\
&&+C\|\nabla u\|_{L^2(\Omega)}^2\|\Delta u\|_{L^2(\Omega)}^2+C\|\nabla B\|_{L^2(\Omega)}^2\|\Delta B\|_{L^2(\Omega)}^2\notag\\
&&+C\|\theta_t\|_{L^2(\Omega)}^2+C\| u_t\|_{L^2(\Omega)}^2.\notag
\een

Thus, by applying Gronwall's inequality, Lemma \ref{L33}-\ref{L34}, Corollary \ref{thetatesti} and Lemma \ref{L35}, it follows that
\beno
&&\|u_t,\,B_t\|_{L^2(\Omega)}^2+C_0^{-1}\int_0^t\|\nabla u_{\tau},\,\nabla B_{\tau}\|_{L^2(\Omega)}^2d\tau\leq C,
\eeno
with the help of which, \eqref{equ353} can be updated as
\ben\label{equ354}
&&\f{d}{dt}\|u_t,\,B_t\|_{L^2(\Omega)}^2+C_0^{-1}\|\nabla u_t,\,\nabla B_t\|_{L^2(\Omega)}^2\notag\\
&\leq&C\|\nabla u\|_{L^2(\Omega)}^2+C\|\nabla B\|_{L^2(\Omega)}^2+C\|\nabla\theta_t\|_{L^2(\Omega)}^2+C\|\Delta u\|_{L^2(\Omega)}^2\\
&&+C\|\Delta B\|_{L^2(\Omega)}^2+C\|\theta_t\|_{L^2(\Omega)}^2+C\| u_t\|_{L^2(\Omega)}^2.\notag
\een
Next, we multiply by $e^{\alpha t}$ on both sides of \eqref{equ354} and integrate in time over $[0,t]$ to derive
\ben\label{equ355}
&&e^{\alpha t}\|u_t,\,B_t\|_{L^2(\Omega)}^2+C_0^{-1}\int_0^te^{\alpha \tau}\|\nabla u_{\tau},\,\nabla B_{\tau}\|_{L^2(\Omega)}^2d\tau\leq C.
\een

At last, recall the equations \eqref{B-MHD-equ-uB}, by employing Lemma \ref{elliptic}-\ref{stokes2}, \ref{L31}-\ref{L35} and \eqref{equ355}, we have
\beno
&&\|\nabla^2u\|_{L^2(\Omega)}+\|\nabla^2B\|_{L^2(\Omega)}+\|\nabla\pi\|_{L^2(\Omega)}\\
&\leq&C\|u_t\|_{L^2(\Omega)}+C\|\mu'(\theta)\nabla\theta\cdot\nabla u\|_{L^2(\Omega)}+C\|u\cdot\nabla u\|_{L^2(\Omega)}+C\|B\cdot\nabla B\|_{L^2(\Omega)}\\
&&+C\|\theta\|_{L^2(\Omega)}+C\|u\|_{H^1(\Omega)}+C\|\pi\|_{L^2(\Omega)}+C\|B_t\|_{L^2(\Omega)}\\
&&+C\|\sigma'(\theta)\grad^{\perp}\theta\grad^{\perp}\cdot B\|_{L^2(\Omega)}+C\|u\cdot\nabla B\|_{L^2(\Omega)}+C\|B\cdot\nabla u\|_{L^2(\Omega)}\\
&\leq&C\|u_t\|_{L^2(\Omega)}+C\|\mu'(\theta)\nabla\theta\cdot\nabla u\|_{L^2(\Omega)}+C\|u\cdot\nabla u\|_{L^2(\Omega)}+C\|B\cdot\nabla B\|_{L^2(\Omega)}\\
&&+C\|\theta\|_{L^2(\Omega)}+C\|u\|_{H^1(\Omega)}+C\|B_t\|_{L^2(\Omega)}+C\|\sigma'(\theta)\grad^{\perp}\theta\grad^{\perp}\cdot B\|_{L^2(\Omega)}\\
&&+C\|u\cdot\nabla B\|_{L^2(\Omega)}+C\|B\cdot\nabla u\|_{L^2(\Omega)}\\
&\leq&C(\|u_t\|_{L^2(\Omega)}+\|B_t\|_{L^2(\Omega)}+\|\theta\|_{L^2(\Omega)}+\|u\|_{H^1(\Omega)})+C\|\nabla \theta\|_{L^4(\Omega)}(\|\nabla u\|_{L^4(\Omega)}\\
&&+\|\nabla B\|_{L^4(\Omega)})+C(\|u\|_{L^4(\Omega)}+\|B\|_{L^4(\Omega)})(\|\nabla u\|_{L^4(\Omega)}+\|\nabla B\|_{L^4(\Omega)})\\
&\leq&C(\|u_t\|_{L^2(\Omega)}+\|B_t\|_{L^2(\Omega)}+\|\theta\|_{L^2(\Omega)}+\|u\|_{H^1(\Omega)})+C(\|u\|_{H^1(\Omega)}+\|B\|_{H^1(\Omega)}\\
&&+\|\theta\|_{H^2(\Omega)})\times\big[\|\nabla u\|_{L^2(\Omega)}+\|\nabla u\|_{L^2(\Omega)}^{\f12}\|\nabla^2 u\|_{L^2(\Omega)}^{\f12}+\|\nabla B\|_{L^2(\Omega)}\\
&&+\|\nabla B\|_{L^2(\Omega)}^{\f12}\|\nabla^2 B\|_{L^2(\Omega)}^{\f12}\big]\\
&\leq&\f12\|\nabla^2u\|_{L^2(\Omega)}+\f12\|\nabla^2B\|_{L^2(\Omega)}+Ce^{-\f{\alpha}2 t},
\eeno
which actually shows $\|\nabla^2u\|_{L^2(\Omega)}^2+\|\nabla^2B\|_{L^2(\Omega)}^2\leq Ce^{-\alpha t}$. By Corollary \ref{stokes3}, Lemma \ref{elliptic} and similar calculation as above, we can also obtain $\int_0^te^{\alpha \tau}\|{u},\,{B}\|_{H^3(\Omega)}^2d\tau\leq C$.
\endproof

\subsection{Proof of Theorem \ref{T2}}\hspace*{\parindent}\\

\no{\bf Proof.}\quad
The existence of strong solutions is from Theorem \ref{T1} and Lemma \ref{L35}-\ref{L36}, the left thing is to show the uniqueness of the solution.\\
\\
{\bf Uniqueness:} For any fixed $T>0$, suppose there are two solutions $(\theta,u,B,\pi)$, $(\widetilde{\theta},\widetilde{u},\widetilde{B},\widetilde{\pi})$ of \eqref{B-MHD-equ} and let $\bar{\theta}=\widetilde{\theta}-\theta,\,\,\bar{u}=\widetilde{u}-u,\,\,\bar{B}=\widetilde{B}-B,\,\,\bar{\pi}=\widetilde{\pi}-\pi$. Then by Remark \ref{gw1}, for any $\phi\in L^2(0,T; H_0^1(\Omega))$ and $\psi,\,\varphi\in L^2(0,T; V)$, $(\bar{\theta},\bar{u},\bar{B},\bar{\pi})$ satisfies the following problem:
\ben\label{unia}
&&\f{d}{dt}<\bar{\theta},\phi>+(\kappa(\widetilde{\theta})\nabla\bar{\theta},\nabla\phi)+([\kappa(\widetilde{\theta})-\kappa(\theta)]\nabla \theta,\nabla\phi)+b(\widetilde{u},\bar{\theta},\phi)\notag\\
&=&-b(\bar{u},\theta,\phi)-(\bar{u}_2,\phi),
\een
\ben\label{unib}
&&\f{d}{dt}<\bar{u},\psi>+(\mu(\widetilde{\theta})\nabla\bar{u},\nabla\psi)+([\mu(\widetilde{\theta})-\mu(\theta)]\nabla u,\nabla\psi)+b(\widetilde{u},\bar{u},\psi)\notag\\
&=&-b(\bar{u},u,\psi)+b(\widetilde{B},\bar{B},\psi)+b(\bar{B},B,\psi)+(\bar{\theta}e_2,\psi),
\een
\ben\label{unic}
&&\f{d}{dt}<\bar{B},\varphi>+(\sigma(\widetilde{\theta})\grad^{\perp}\cdot \bar{B},\grad^{\perp}\cdot\varphi)-([\sigma(\widetilde{\theta})-\sigma(\theta)]\grad^{\perp}\cdot {B},\grad^{\perp}\cdot\varphi)\notag\\
&=&-b(\widetilde{u},\bar{B},\psi)-b(\bar{u},B,\psi)+b(\widetilde{B},\bar{u},\psi)+b(\bar{B},u,\psi),
\een
\ben\label{unid}
(\bar{\theta},\bar{u},\bar{B})|_{\p\Omega}=0,\,\,\,(\bar{\theta},\bar{u},\bar{B})(x,0)=0.
\een

Taking $\phi=\bar{\theta}$ in $(\ref{unia})$, $\psi=\bar{u}$ in $(\ref{unib})$, $\varphi=\bar{B}$ in $(\ref{unic})$  respectively, and applying Lemma \ref{curlB}, it follows that
\ben\label{uni0}
&&\f12\f{d}{dt}\|\bar{\theta},\,\bar{u},\,\bar{B}\|_{L^2(\Omega)}^2+C_0^{-1}\|\nabla\bar{\theta},\,\nabla \bar{u},\,\nabla\bar{B}\|_{L^2(\Omega)}^2\notag\\
&\leq&-\int_{\Omega}[\kappa(\widetilde{\theta})-\kappa(\theta)]\nabla\theta\cdot\nabla\bar{\theta}dx-\int_{\Omega}\bar{u}\cdot\nabla \theta\bar{\theta}dx-\int_{\Omega}\bar{u}\cdot\nabla u\cdot\bar{u}dx\\
&&-\int_{\Omega}[\mu(\widetilde{\theta})-\mu(\theta)]\nabla u\cdot\nabla\bar{u}\,dx-\int_{\Omega}\bar{B}\cdot\nabla B\cdot\bar{u}dx-\int_{\Omega}\bar{u}\cdot\nabla B\cdot\bar{B}dx\notag\\
&&-\int_{\Omega}[\sigma(\widetilde{\theta})-\sigma(\theta)]\grad^{\perp}\cdot B\grad^{\perp}\cdot \bar{B}dx-\int_{\Omega}\bar{B}\cdot\nabla u\cdot\bar{B}dx\notag\\
&=&\sum\limits_{j=1}^8I_j,\notag
\een
where we make use of (\ref{b2}), the fact $(\bar{\theta}e_2,\bar{u})-(\bar{u}_2,\bar{\theta})=0$, $b(\widetilde{B},\bar{B},\bar{u})+b(\widetilde{B},\bar{u},\bar{B})=0$ and Lions-Magenes Lemma (see e.g., \cite{RT}).
In the following, we will deal with the eight terms one by one. Firstly, by (\ref{M}), it is clear to get
\ben\label{uni00}
&&\|\kappa(\widetilde{\theta})-\kappa(\theta)\|_{L^2(\Omega)}+\|\mu(\widetilde{\theta})-\mu(\theta)\|_{L^2(\Omega)}+\|\sigma(\widetilde{\theta})-\sigma(\theta)\|_{L^2(\Omega)}\notag\\
&&+\|\kappa'(\widetilde{\theta})-\kappa'(\theta)\|_{L^2(\Omega)}+\|\mu'(\widetilde{\theta})-\mu'(\theta)\|_{L^2(\Omega)}+\|\sigma'(\widetilde{\theta})-\sigma'(\theta)\|_{L^2(\Omega)}\notag\\
&\leq&\|\int_{\theta}^{\widetilde{\theta}}\kappa'(s)\,ds\|_{L^2(\Omega)}+\|\int_{\theta}^{\widetilde{\theta}}\mu'(s)\,ds\|_{L^2(\Omega)}+\|\int_{\theta}^{\widetilde{\theta}}\sigma'(s)\,ds\|_{L^2(\Omega)}\\
&&+\|\int_{\theta}^{\widetilde{\theta}}\kappa''(s)\,ds\|_{L^2(\Omega)}+\|\int_{\theta}^{\widetilde{\theta}}\mu''(s)\,ds\|_{L^2(\Omega)}+\|\int_{\theta}^{\widetilde{\theta}}\sigma''(s)\,ds\|_{L^2(\Omega)}\notag\\
&\leq&M\|\bar{\theta}\|_{L^2(\Omega)}.\notag
\een

Then by H\"{o}lder inequality, it holds that
\ben\label{uni1}
&&I_1+I_4+I_7\notag\\
&\leq&\|\nabla\bar{\theta}\|_{L^2(\Omega)}\|\nabla{\theta}\|_{L^4(\Omega)}\|\kappa(\widetilde{\theta})-\kappa(\theta)\|_{L^4(\Omega)}+
\|\nabla\bar{u}\|_{L^2(\Omega)}\|\nabla{u}\|_{L^4(\Omega)}\|\mu(\widetilde{\theta})-\mu(\theta)\|_{L^4(\Omega)}\notag\\
&&+\|\nabla\bar{B}\|_{L^2(\Omega)}\|\nabla{B}\|_{L^4(\Omega)}\|\sigma(\widetilde{\theta})-\sigma(\theta)\|_{L^4(\Omega)}.
\een
On the other hand,  by use of (\ref{uni00}), Corollary \ref{C1}, $\|\kappa(\widetilde{\theta})-\kappa(\theta)\|_{L^4(\Omega)}$ can be estimated as
\beno
&&\|\kappa(\widetilde{\theta})-\kappa(\theta)\|_{L^4(\Omega)}\notag\\
&\leq&\|\kappa(\widetilde{\theta})-\kappa(\theta)\|_{L^2(\Omega)}^{\f12}\|\kappa'(\widetilde{\theta})\nabla\widetilde{\theta}-\kappa'(\theta)\nabla\theta\|_{L^2(\Omega)}^{\f12}+\|\kappa(\widetilde{\theta})-\kappa(\theta)\|_{L^2(\Omega)}\notag\\
&\leq&C\|\bar{\theta}\|_{L^2(\Omega)}^{\f12}(\|\kappa'(\widetilde{\theta})\nabla\bar{\theta}\|_{L^{2(\Omega)}}^{\f12}+\|(\kappa'(\widetilde{\theta})-\kappa'(\theta))\nabla\theta\|_{L^{2}(\Omega)}^{\f12})+\|\bar{\theta}\|_{L^2(\Omega)}\\
&\leq&C\|\bar{\theta}\|_{L^2(\Omega)}^{\f12}(\|\kappa'(\widetilde{\theta})\|_{L^\infty(\Omega)}^{\f12}\|\nabla\bar{\theta}\|_{L^{2}(\Omega)}^{\f12}
+\|\bar{\theta}\|_{L^{2}(\Omega)}^{\f12}\|\nabla\theta\|_{L^{\infty}(\Omega)}^{\f12})+\|\bar{\theta}\|_{L^2(\Omega)}\notag\\
&\leq&C\|\bar{\theta}\|_{L^2(\Omega)}^{\f12}\|\nabla\bar{\theta}\|_{L^{2}(\Omega)}^{\f12}
+\|\bar{\theta}\|_{L^{2}(\Omega)}\|\nabla\theta\|_{L^{\infty}(\Omega)}^{\f12}+\|\bar{\theta}\|_{L^2(\Omega)}\notag\\
&\leq&C\|\bar{\theta}\|_{L^2(\Omega)}^{\f12}\|\nabla\bar{\theta}\|_{L^{2}(\Omega)}^{\f12}
+\|\bar{\theta}\|_{L^{2}(\Omega)}\|\theta\|_{H^{3}(\Omega)}^{\f12}+\|\bar{\theta}\|_{L^2(\Omega)}\notag,
\eeno
which also holds for $\|\mu(\widetilde{\theta})-\mu(\theta)\|_{L^4(\Omega)}$ and $\|\sigma(\widetilde{\theta})-\sigma(\theta)\|_{L^4(\Omega)}$. Thus, by Lemma \ref{L31}-\ref{L36}, we can update \eqref{uni1} as
\ben\label{uni2}
&&I_1+I_4+I_7\notag\\
&\leq&C(\|\bar{\theta}\|_{L^2(\Omega)}^{\f12}\|\nabla\bar{\theta}\|_{L^{2}(\Omega)}^{\f12}
+\|\bar{\theta}\|_{L^{2}(\Omega)}\|\theta\|_{H^{3}(\Omega)}^{\f12}+\|\bar{\theta}\|_{L^2(\Omega)})\big[\|\nabla\bar{\theta}\|_{L^2(\Omega)}(\|\nabla{\theta}\|_{L^2(\Omega)}\notag\\
&&+\|\nabla{\theta}\|_{L^2(\Omega)}^{\f12}\|\Delta{\theta}\|_{L^2(\Omega)}^{\f12})+\|\nabla\bar{u}\|_{L^2(\Omega)}(\|\nabla{u}\|_{L^2(\Omega)}^{\f12}\|\Delta{u}\|_{L^2(\Omega)}^{\f12}+\|\nabla{u}\|_{L^2(\Omega)})\notag\\
&&+\|\nabla\bar{B}\|_{L^2(\Omega)}(\|\nabla{B}\|_{L^2(\Omega)}^{\f12}\|\Delta{B}\|_{L^2(\Omega)}^{\f12}+\|\nabla{B}\|_{L^2(\Omega)})\big]\\
&\leq&Ce^{-\f{\alpha t}{2}}\|\nabla\bar{\theta},\,\nabla\bar{u},\,\nabla\bar{B}\|_{L^2(\Omega)}(\|\bar{\theta}\|_{L^2(\Omega)}^{\f12}\|\nabla\bar{\theta}\|_{L^{2}(\Omega)}^{\f12}
+\|\bar{\theta}\|_{L^{2}(\Omega)}\|\theta\|_{H^{3}(\Omega)}^{\f12}+\|\bar{\theta}\|_{L^2(\Omega)})\notag\\
&\leq&\f14\|\nabla\bar{\theta},\,\nabla\bar{u},\,\nabla\bar{B}\|_{L^2(\Omega)}^2+Ce^{-\alpha t}(\|\bar{\theta}\|_{L^2(\Omega)}^2
+\|\theta\|_{H^{3}(\Omega)}\|\bar{\theta}\|_{L^{2}(\Omega)}^2).\notag
\een

To estimate the left terms, by employing Corollary \ref{C1}, Lemma \ref{L31}-\ref{L36}, one can get
\ben\label{uni3}
&&I_2+I_3+I_5+I_8\notag\\
&\leq&\|\nabla{\theta},\,\nabla{u},\,\nabla{B}\|_{L^2(\Omega)}\|\bar{\theta},\,\bar{u},\,\bar{B}\|_{L^4(\Omega)}^2\notag\\
&\leq&Ce^{-\f{\alpha t}{2}}\|\bar{\theta},\,\bar{u},\,\bar{B}\|_{L^2(\Omega)}\|\nabla\bar{\theta},\,\nabla\bar{u},\,\nabla\bar{B}\|_{L^2(\Omega)}\\
&\leq&\f14\|\nabla\bar{\theta},\,\nabla\bar{u},\,\nabla\bar{B}\|_{L^2(\Omega)}^2+Ce^{-\alpha t}\|\bar{\theta},\,\bar{u},\,\bar{B}\|_{L^2(\Omega)}^2\notag.
\een

Thus, by summing up (\ref{uni0}), (\ref{uni2})-(\ref{uni3}) and using Young inequality, we have
\beno
&&\f{d}{dt}\|\bar{\theta},\,\bar{u},\,\bar{B}\|_{L^2(\Omega)}^2+C_0^{-1}\|\nabla\bar{\theta},\,\nabla \bar{u},\,\nabla\bar{B}\|_{L^2(\Omega)}^2\\
&\leq&C(e^{-\alpha t}+e^{\alpha t}\|\theta\|_{H^{3}(\Omega)}^2)\|\bar{\theta},\,\bar{u},\,\bar{B}\|_{L^2(\Omega)}^2,
\eeno
which implies, after applying Gronwall's inequality and Lemma \ref{L35}, that
\beno
\|\bar{\theta}\|_{L^2(\Omega)}^2+\|\bar{u}\|_{L^2(\Omega)}^2+\|\bar{B}\|_{L^2(\Omega)}^2\leq C(\|\bar{\theta}_0\|_{L^2(\Omega)}^2+\|\bar{u}_0\|_{L^2(\Omega)}^2+\|\bar{B}_0\|_{L^2(\Omega)}^2)=0
\eeno
for any $t\in[0,T].$ Thus we finish all the proof.
\endproof

\section*{Acknowledgments}
D. Bian is partially supported by the National Natural Science Foundation of China (Nos. 11501028 and 11471323), the Postdoctoral Science Foundation of China (No. 2016T90038), and the Basic Research Foundation of Beijing Institute of Technology (20151742001). J. Liu is supported by the Connotation Development Funds of Beijing University of Technology (No. 006000514116041).
\par

\end{document}